\newif\ifdraft
\draftfalse

\documentclass[reqno,a4paper]{amsart}

\usepackage{mathrsfs,amsmath,amssymb,graphicx}
\usepackage{stmaryrd}
\usepackage[unicode]{hyperref}
\usepackage{latexsym}
\usepackage{layout}
\usepackage[english]{babel}
\usepackage{color}
\usepackage{subcaption}
\usepackage{fancyvrb}
\usepackage{color}
\usepackage{amsmath,amscd}
\usepackage{multicol}
\usepackage{multirow}
\usepackage[percent]{overpic}
\usepackage{mathtools}
\usepackage{tikz} 
\usepackage{overpic}
\usepackage{enumitem} 
\usepackage{transparent}

\usetikzlibrary{matrix,arrows,decorations,shapes,calc}
\tikzset{every picture/.append style={remember picture},
na/.style={baseline=-.5ex}}

\theoremstyle{plain}
\newtheorem{theorem}{Theorem}[section]

\newtheorem{corollary}[theorem]{Corollary}

\theoremstyle{definition}
\newtheorem{definition}{Definition}[section]
\theoremstyle{remark}
\newtheorem{remark}{Remark}[section]

\newcommand{\nada}[1]   {}

\newcommand{\graph}{{\rm G}}

\newcommand{\pair} {(\mathbb S^3,V)}
\newcommand{\pairprime} {(\mathbb S^3,V')}

\newcommand{\Sbb} {\mathbb S}

\newcommand{\sphere} {\Sbb^3}

\newcommand{\Z} {\mathbb Z}
\definecolor{mygray}{rgb}{0.92,0.92,0.92}

\newcommand{\Compl}[1]{E(#1)}
\newcommand{\rnbhd}[1]{\mathcal{N}(#1)}

\newcommand{\onesum}{\operatorname{-{}-}}
\newcommand{\twosum}{\operatorname{\#}}
\newcommand{\op}[1]{\operatorname{#1}}

\newcommand{\Kzero}{\operatorname{K0}_1}
\newcommand{\Kthree}{\operatorname{K3}_1}
\newcommand{\Kfour}{\operatorname{K4}_1}
\newcommand{\Kfiveone}{\operatorname{K5}_1}
\newcommand{\Kfivetwo}{\operatorname{K5}_2}
\newcommand{\Ksevenone}{\operatorname{K7}_1}
\newcommand{\Kseventwo}{\operatorname{K7}_2}
\newcommand{\Kseventhree}{\operatorname{K7}_3}
\newcommand{\Ksevenfour}{\operatorname{K7}_4}
\newcommand{\Ksevenfive}{\operatorname{K7}_5}
\newcommand{\Ksevensix}{\operatorname{K7}_6}
\newcommand{\Ksevenseven}{\operatorname{K7}_7}

\newcommand{\fiveone}{{5_1}}

\newcommand{\seventhirtynine}{7_{39}}
\newcommand{\sevenforty}{7_{40}}
\newcommand{\sevenfortyone}{7_{41}}
\newcommand{\sevenfortythree}{7_{43}}
\newcommand{\sevenfortyfour}{7_{44}}
\newcommand{\sevenfiftytwo}{7_{52}}
\newcommand{\sevenfiftynine}{7_{59}}
\newcommand{\sevensixty}{7_{60}}
\newcommand{\sevensixtyone}{7_{61}}
\newcommand{\sevensixtytwo}{7_{62}}
\newcommand{\sevensixtythree}{7_{63}}
\newcommand{\sevensixtyfour}{7_{64}}
\newcommand{\sevensixtyfive}{7_{65}}
\newcommand{\sevensixtysix}{7_{66}}
\newcommand{\sevensixtyseven}{7_{67}}
\newcommand{\sevensixtyeight}{7_{68}}
\newcommand{\sevensixtynine}{7_{69}}

\newcommand{\fffbox}{\makebox}

\newcommand{\draftMP}[1]{\ifdraft{\color{blue}#1}\fi}
\newcommand{\draftGB}[1]{\ifdraft{\color{red}#1}\fi}
\newcommand{\draftYW}[1]{\ifdraft{\color{orange}#1}\fi}
\newcommand{\draftGP}[1]{\ifdraft{\color{teal}#1}\fi}

\graphicspath{{figures/}}

\numberwithin{equation}{section}
\numberwithin{figure}{section}

\title{A table 
of genus two handlebody-knots 
with seven crossings}
\author[G.\ Bellettini]{Giovanni Bellettini}
\address{Dipartimento di Scienze Matematiche, Informatiche e Fisiche,
via delle Scienze 206, 33100 Udine UD, Italy
and International Centre for Theoretical Physics ICTP,
Mathematics Section, 34151 Trieste, Italy
}
\email{giovanni.bellettini@uniud.it}

\author[G.\ Paolini]{Giovanni Paolini}
\address{Dipartimento di Matematica, Universit\`a di Bologna, Piazza di Porta
San Donato 5, 40126 Bologna, Italy}
\email{g.paolini@unibo.it}

\author[M.\ Paolini]{Maurizio Paolini}
\address{Dipartimento di Matematica e Fisica, Universit\`a Cattolica del Sacro Cuore, 25121 Brescia, Italy}
\email{maurizio.paolini@unicatt.it}

\author[Y.\ S.\ Wang]{Yi-Sheng Wang}
\address{National Sun Yat-sen University, Kaohsiung, Taiwan}
\email{yisheng@math.nsysu.edu.tw}

\date{\today}

\begin{document}

\thanks{}

\begin{abstract}
We enumerate all genus two handlebody-knots with seven crossings, up to mirror image, 
extending the Ishii-Kishimoto-Moriuchi-Suzuki table.
\end{abstract}

\maketitle

%
%

\draftMP{DraftMP comments by Maurizio}
\draftYW{DraftYW comments by Yi-Sheng}
\draftGB{DraftGB comments by Giovanni B}
\draftGP{DraftGP comments by Giove}

\section{Introduction}\label{sec:intro}

A \emph{genus $g$ handlebody-knot} is an embedded handlebody $V$ of genus $g$ in the oriented $3$-sphere $\sphere$, and a 
\emph{spatial graph} is an embedded finite graph 
$\Gamma$ in $\sphere$. A spine of a handlebody-knot is a spatial graph, and a regular neighborhood of a spatial graph is a handlebody-knot. While every spatial graph has a unique regular neighborhood, up to isotopy, 
a handlebody-knot can have infinitely many different spines.
Two handlebody-knots $V,V'$ (resp.\ spatial graphs $\Gamma,\Gamma'$) are \emph{equivalent} if there is an orientation-preserving self-homeomorphism of $\sphere$ sending $V$ to $V'$ (resp.\ sending $\Gamma$ to $\Gamma'$); if the condition ``orientation-preserving'' is dropped, we say they are \emph{equivalent, up to mirror image}.  
The \emph{crossing number} of a handlebody-knot $V$ is the minimum of the crossing numbers of all its spines.
A handlebody-knot $V$ is \emph{irreducible} if there exists no $2$-sphere in $\sphere$ meeting 
$V$ at an essential disk of $V$.

Ishii-Kishimoto-Moriuchi-Suzuki \cite[Table 1]{IshKisMorSuz:12} enumerates, up to mirror image, all irreducible genus two handlebody-knots, up to six crossings. Our main result extends their table to seven crossing genus two handlebody-knots.
 
\begin{theorem}\label{teo:irred}
Table {\rm \ref{tab:table}} enumerates all irreducible genus two handlebody-knots with seven crossings, up to mirror image.
\end{theorem}

The construction of Table \ref{tab:table} allows us to obtain a table of reducible handlebody-knots, extending \cite[Table $4$]{IshKisMorSuz:12}.

\begin{theorem}\label{teo:red}
Table {\rm \ref{tab:tablered}} enumerates all reducible genus two handlebody-knots with seven crossings, up to mirror image.
\end{theorem}

The paper is structured so that the completeness of the two tables is given in Section \ref{sec:completeness},
while in Sections \ref{sec:uniqueness} and \ref{sec:hardpairs}, we prove that there are no duplications in Tables \ref{tab:table}, \ref{tab:tablered}.
%
Section \ref{sec:irreducibility} explains how the irreducibility of handlebody-knots in Table \ref{tab:table} is obtained.






\newdimen\entrysize \entrysize=59pt

\newcommand{\tabentry}[1]{\fffbox[\entrysize]{\includegraphics[height=0.15\textwidth]{figures/t#1}}$7_{#1}$}
\newcommand{\tabentrym}[1]{\fffbox[\entrysize]{\includegraphics[height=0.15\textwidth]{figures/tmissing}}$7_{#1}$}
\ifdraft
\newcommand{\tabentryt}[1]{\fffbox[\entrysize]{\begin{tikzpicture}[scale=0.8,every edge/.style={draw,thick}]
\draw[green,thin] (-1.05,-1.05) rectangle (1.05,1.05);
\input{tikz/#1.tikz}\end{tikzpicture}}$7_{#1}$}\!
\else
\newcommand{\tabentryt}[1]{\fffbox[\entrysize]{\begin{tikzpicture}[scale=0.8,every edge/.style={draw,thick}]\input{tikz/#1.tikz}\end{tikzpicture}}\!$7_{#1}$}
\fi

\newcounter{tableone}
\setcounter{tableone}{\value{table}}

\begin{table}[ht]
\caption{Irreducible genus two handlebody-knots with seven crossings (continued on next page).}
\label{tab:table}
\fffbox[\entrysize]{\begin{tikzpicture}[scale=0.8,every edge/.style={draw,thick}]\input{tikz/1.tikz}\end{tikzpicture}}\!$7_{1}$%
\fffbox[\entrysize]{\begin{tikzpicture}[scale=0.8,every edge/.style={draw,thick}]\input{tikz/2.tikz}\end{tikzpicture}}\!$7_{2}$%
\fffbox[\entrysize]{\begin{tikzpicture}[scale=0.8,every edge/.style={draw,thick}]\input{tikz/3.tikz}\end{tikzpicture}}\!$7_{3}$%
\fffbox[\entrysize]{\begin{tikzpicture}[scale=0.8,every edge/.style={draw,thick}]\input{tikz/4.tikz}\end{tikzpicture}}\!$7_{4}$%
\fffbox[\entrysize]{\begin{tikzpicture}[scale=0.8,every edge/.style={draw,thick}]\input{tikz/5.tikz}\end{tikzpicture}}\!$7_{5}$%
\\
\smallskip
\fffbox[\entrysize]{\begin{tikzpicture}[scale=0.8,every edge/.style={draw,thick}]\input{tikz/6.tikz}\end{tikzpicture}}\!$7_{6}$%
\fffbox[\entrysize]{\begin{tikzpicture}[scale=0.8,every edge/.style={draw,thick}]\input{tikz/7.tikz}\end{tikzpicture}}\!$7_{7}$%
\fffbox[\entrysize]{\begin{tikzpicture}[scale=0.8,every edge/.style={draw,thick}]\input{tikz/8.tikz}\end{tikzpicture}}\!$7_{8}$%
\fffbox[\entrysize]{\begin{tikzpicture}[scale=0.8,every edge/.style={draw,thick}]\input{tikz/9.tikz}\end{tikzpicture}}\!$7_{9}$%
\fffbox[\entrysize]{\begin{tikzpicture}[scale=0.8,every edge/.style={draw,thick}]\input{tikz/10.tikz}\end{tikzpicture}}\!$7_{10}$%
\\
\smallskip
\fffbox[\entrysize]{\begin{tikzpicture}[scale=0.8,every edge/.style={draw,thick}]\input{tikz/11.tikz}\end{tikzpicture}}\!$7_{11}$%
\fffbox[\entrysize]{\begin{tikzpicture}[scale=0.8,every edge/.style={draw,thick}]\input{tikz/12.tikz}\end{tikzpicture}}\!$7_{12}$%
\fffbox[\entrysize]{\begin{tikzpicture}[scale=0.8,every edge/.style={draw,thick}]\input{tikz/13.tikz}\end{tikzpicture}}\!$7_{13}$%
\fffbox[\entrysize]{\begin{tikzpicture}[scale=0.8,every edge/.style={draw,thick}]\input{tikz/14.tikz}\end{tikzpicture}}\!$7_{14}$%
\fffbox[\entrysize]{\begin{tikzpicture}[scale=0.8,every edge/.style={draw,thick}]\input{tikz/15.tikz}\end{tikzpicture}}\!$7_{15}$%
\\
\smallskip
\fffbox[\entrysize]{\begin{tikzpicture}[scale=0.8,every edge/.style={draw,thick}]\input{tikz/16.tikz}\end{tikzpicture}}\!$7_{16}$%
\fffbox[\entrysize]{\begin{tikzpicture}[scale=0.8,every edge/.style={draw,thick}]\input{tikz/17.tikz}\end{tikzpicture}}\!$7_{17}$%
\fffbox[\entrysize]{\begin{tikzpicture}[scale=0.8,every edge/.style={draw,thick}]\input{tikz/18.tikz}\end{tikzpicture}}\!$7_{18}$%
\fffbox[\entrysize]{\begin{tikzpicture}[scale=0.8,every edge/.style={draw,thick}]\input{tikz/19.tikz}\end{tikzpicture}}\!$7_{19}$%
\fffbox[\entrysize]{\begin{tikzpicture}[scale=0.8,every edge/.style={draw,thick}]\input{tikz/20.tikz}\end{tikzpicture}}\!$7_{20}$%
\\
\smallskip
\fffbox[\entrysize]{\begin{tikzpicture}[scale=0.8,every edge/.style={draw,thick}]\input{tikz/21.tikz}\end{tikzpicture}}\!$7_{21}$%
\fffbox[\entrysize]{\begin{tikzpicture}[scale=0.8,every edge/.style={draw,thick}]\input{tikz/22.tikz}\end{tikzpicture}}\!$7_{22}$%
\fffbox[\entrysize]{\begin{tikzpicture}[scale=0.8,every edge/.style={draw,thick}]\input{tikz/23.tikz}\end{tikzpicture}}\!$7_{23}$%
\fffbox[\entrysize]{\begin{tikzpicture}[scale=0.8,every edge/.style={draw,thick}]\input{tikz/24.tikz}\end{tikzpicture}}\!$7_{24}$%
\fffbox[\entrysize]{\begin{tikzpicture}[scale=0.8,every edge/.style={draw,thick}]\input{tikz/25.tikz}\end{tikzpicture}}\!$7_{25}$%
\\
\smallskip
\fffbox[\entrysize]{\begin{tikzpicture}[scale=0.8,every edge/.style={draw,thick}]\input{tikz/26.tikz}\end{tikzpicture}}\!$7_{26}$%
\fffbox[\entrysize]{\begin{tikzpicture}[scale=0.8,every edge/.style={draw,thick}]\input{tikz/27.tikz}\end{tikzpicture}}\!$7_{27}$%
\fffbox[\entrysize]{\begin{tikzpicture}[scale=0.8,every edge/.style={draw,thick}]\input{tikz/28.tikz}\end{tikzpicture}}\!$7_{28}$%
\fffbox[\entrysize]{\begin{tikzpicture}[scale=0.8,every edge/.style={draw,thick}]\input{tikz/29.tikz}\end{tikzpicture}}\!$7_{29}$%
\fffbox[\entrysize]{\begin{tikzpicture}[scale=0.8,every edge/.style={draw,thick}]\input{tikz/30.tikz}\end{tikzpicture}}\!$7_{30}$%
\\
\smallskip
\fffbox[\entrysize]{\begin{tikzpicture}[scale=0.8,every edge/.style={draw,thick}]\input{tikz/31.tikz}\end{tikzpicture}}\!$7_{31}$%
\fffbox[\entrysize]{\begin{tikzpicture}[scale=0.8,every edge/.style={draw,thick}]\input{tikz/32.tikz}\end{tikzpicture}}\!$7_{32}$%
\fffbox[\entrysize]{\begin{tikzpicture}[scale=0.8,every edge/.style={draw,thick}]\input{tikz/33.tikz}\end{tikzpicture}}\!$7_{33}$%
\fffbox[\entrysize]{\begin{tikzpicture}[scale=0.8,every edge/.style={draw,thick}]\input{tikz/34.tikz}\end{tikzpicture}}\!$7_{34}$%
\fffbox[\entrysize]{\begin{tikzpicture}[scale=0.8,every edge/.style={draw,thick}]\input{tikz/35.tikz}\end{tikzpicture}}\!$7_{35}$%
\\
\smallskip
\fffbox[\entrysize]{\begin{tikzpicture}[scale=0.8,every edge/.style={draw,thick}]\input{tikz/36.tikz}\end{tikzpicture}}\!$7_{36}$%
\fffbox[\entrysize]{\begin{tikzpicture}[scale=0.8,every edge/.style={draw,thick}]\input{tikz/37.tikz}\end{tikzpicture}}\!$7_{37}$%
\fffbox[\entrysize]{\begin{tikzpicture}[scale=0.8,every edge/.style={draw,thick}]\input{tikz/38.tikz}\end{tikzpicture}}\!$7_{38}$%
\fffbox[\entrysize]{\begin{tikzpicture}[scale=0.8,every edge/.style={draw,thick}]\input{tikz/39.tikz}\end{tikzpicture}}\!$7_{39}$%
\fffbox[\entrysize]{\begin{tikzpicture}[scale=0.8,every edge/.style={draw,thick}]\input{tikz/40.tikz}\end{tikzpicture}}\!$7_{40}$%
\end{table}
\setcounter{table}{\value{tableone}}
\begin{table}[ht]
\caption{(continued)}
\fffbox[\entrysize]{\begin{tikzpicture}[scale=0.8,every edge/.style={draw,thick}]\input{tikz/41.tikz}\end{tikzpicture}}\!$7_{41}$%
\fffbox[\entrysize]{\begin{tikzpicture}[scale=0.8,every edge/.style={draw,thick}]\input{tikz/42.tikz}\end{tikzpicture}}\!$7_{42}$%
\fffbox[\entrysize]{\begin{tikzpicture}[scale=0.8,every edge/.style={draw,thick}]\input{tikz/43.tikz}\end{tikzpicture}}\!$7_{43}$%
\fffbox[\entrysize]{\begin{tikzpicture}[scale=0.8,every edge/.style={draw,thick}]\input{tikz/44.tikz}\end{tikzpicture}}\!$7_{44}$%
\fffbox[\entrysize]{\begin{tikzpicture}[scale=0.8,every edge/.style={draw,thick}]\input{tikz/45.tikz}\end{tikzpicture}}\!$7_{45}$%
\\
\smallskip
\fffbox[\entrysize]{\begin{tikzpicture}[scale=0.8,every edge/.style={draw,thick}]\input{tikz/46.tikz}\end{tikzpicture}}\!$7_{46}$%
\fffbox[\entrysize]{\begin{tikzpicture}[scale=0.8,every edge/.style={draw,thick}]\input{tikz/47.tikz}\end{tikzpicture}}\!$7_{47}$%
\fffbox[\entrysize]{\begin{tikzpicture}[scale=0.8,every edge/.style={draw,thick}]\input{tikz/48.tikz}\end{tikzpicture}}\!$7_{48}$%
\fffbox[\entrysize]{\begin{tikzpicture}[scale=0.8,every edge/.style={draw,thick}]\input{tikz/49.tikz}\end{tikzpicture}}\!$7_{49}$%
\fffbox[\entrysize]{\begin{tikzpicture}[scale=0.8,every edge/.style={draw,thick}]\input{tikz/50.tikz}\end{tikzpicture}}\!$7_{50}$%
\\
\smallskip
\fffbox[\entrysize]{\begin{tikzpicture}[scale=0.8,every edge/.style={draw,thick}]\input{tikz/51.tikz}\end{tikzpicture}}\!$7_{51}$%
\fffbox[\entrysize]{\begin{tikzpicture}[scale=0.8,every edge/.style={draw,thick}]\input{tikz/52.tikz}\end{tikzpicture}}\!$7_{52}$%
\fffbox[\entrysize]{\begin{tikzpicture}[scale=0.8,every edge/.style={draw,thick}]\input{tikz/53.tikz}\end{tikzpicture}}\!$7_{53}$%
\fffbox[\entrysize]{\begin{tikzpicture}[scale=0.8,every edge/.style={draw,thick}]\input{tikz/54.tikz}\end{tikzpicture}}\!$7_{54}$%
\fffbox[\entrysize]{\begin{tikzpicture}[scale=0.8,every edge/.style={draw,thick}]\input{tikz/55.tikz}\end{tikzpicture}}\!$7_{55}$%
\\
\smallskip
\fffbox[\entrysize]{\begin{tikzpicture}[scale=0.8,every edge/.style={draw,thick}]\input{tikz/56.tikz}\end{tikzpicture}}\!$7_{56}$%
\fffbox[\entrysize]{\begin{tikzpicture}[scale=0.8,every edge/.style={draw,thick}]\input{tikz/57.tikz}\end{tikzpicture}}\!$7_{57}$%
\fffbox[\entrysize]{\begin{tikzpicture}[scale=0.8,every edge/.style={draw,thick}]\input{tikz/58.tikz}\end{tikzpicture}}\!$7_{58}$%
\fffbox[\entrysize]{\begin{tikzpicture}[scale=0.8,every edge/.style={draw,thick}]\input{tikz/59.tikz}\end{tikzpicture}}\!$7_{59}$%
\fffbox[\entrysize]{\begin{tikzpicture}[scale=0.8,every edge/.style={draw,thick}]\input{tikz/60.tikz}\end{tikzpicture}}\!$7_{60}$%
\\
\smallskip
\fffbox[\entrysize]{\begin{tikzpicture}[scale=0.8,every edge/.style={draw,thick}]\input{tikz/61.tikz}\end{tikzpicture}}\!$7_{61}$%
\fffbox[\entrysize]{\begin{tikzpicture}[scale=0.8,every edge/.style={draw,thick}]\input{tikz/62.tikz}\end{tikzpicture}}\!$7_{62}$%
\fffbox[\entrysize]{\begin{tikzpicture}[scale=0.8,every edge/.style={draw,thick}]\input{tikz/63.tikz}\end{tikzpicture}}\!$7_{63}$%
\fffbox[\entrysize]{\begin{tikzpicture}[scale=0.8,every edge/.style={draw,thick}]\input{tikz/64.tikz}\end{tikzpicture}}\!$7_{64}$%
\fffbox[\entrysize]{\begin{tikzpicture}[scale=0.8,every edge/.style={draw,thick}]\input{tikz/65.tikz}\end{tikzpicture}}\!$7_{65}$%
\\
\smallskip
\fffbox[\entrysize]{\begin{tikzpicture}[scale=0.8,every edge/.style={draw,thick}]\input{tikz/66.tikz}\end{tikzpicture}}\!$7_{66}$%
\fffbox[\entrysize]{\begin{tikzpicture}[scale=0.8,every edge/.style={draw,thick}]\input{tikz/67.tikz}\end{tikzpicture}}\!$7_{67}$%
\fffbox[\entrysize]{\begin{tikzpicture}[scale=0.8,every edge/.style={draw,thick}]\input{tikz/68.tikz}\end{tikzpicture}}\!$7_{68}$%
\fffbox[\entrysize]{\begin{tikzpicture}[scale=0.8,every edge/.style={draw,thick}]\input{tikz/69.tikz}\end{tikzpicture}}\!$7_{69}$%
\end{table}

\newdimen\entrysize \entrysize=75pt

\newcommand{\redentry}[1]{\fffbox[\entrysize]{\includegraphics[height=0.15\textwidth]{figures/r#1}}}
\newcommand{\redentrym}[1]{\fffbox[\entrysize]{\includegraphics[height=0.15\textwidth]{figures/rmissing}}}
\ifdraft
\newcommand{\redentryt}[2]{\fffbox[\entrysize]{\vbox{\begin{tikzpicture}[scale=0.8,every edge/.style={draw,thick}]
\draw[green,thin] (-1,-1) rectangle (1,1);
\input{tikz/#1.tikz}\end{tikzpicture} \\ #2}}}
\newcommand{\redentrytm}[2]{\fffbox[\entrysize]{\vbox{\begin{tikzpicture}[scale=0.8,every edge/.style={draw,thick}]
\begin{scope}[yscale=-1]
\draw[green,thin] (-1,-1) rectangle (1,1);
\input{tikz/#1.tikz}\end{scope}\end{tikzpicture} \\ #2}}}
\else
\newcommand{\redentryt}[2]{\fffbox[\entrysize]{\vbox{\begin{tikzpicture}[scale=0.8,every edge/.style={draw,thick}]\input{tikz/#1.tikz}\end{tikzpicture} \\ #2}}}
\newcommand{\redentrytm}[2]{\fffbox[\entrysize]{\vbox{\begin{tikzpicture}[scale=0.8,every edge/.style={draw,thick}]
\begin{scope}[yscale=-1]\input{tikz/#1.tikz}\end{scope}\end{tikzpicture} \\ #2}}}
\fi

\begin{table}[ht]
\caption{Reducible handlebody-knots with seven crossings.}
\label{tab:tablered}
\fffbox[\entrysize]{\vbox{\begin{tikzpicture}[scale=0.8,every edge/.style={draw,thick}]
\begin{scope}[yscale=-1]\input{tikz/r1.tikz}\end{scope}\end{tikzpicture} \\ $\Ksevenone \onesum \Kzero$}}%
\fffbox[\entrysize]{\vbox{\begin{tikzpicture}[scale=0.8,every edge/.style={draw,thick}]\input{tikz/r2.tikz}\end{tikzpicture} \\ $\Kseventwo \onesum \Kzero$}}%
\fffbox[\entrysize]{\vbox{\begin{tikzpicture}[scale=0.8,every edge/.style={draw,thick}]\input{tikz/r3.tikz}\end{tikzpicture} \\ $\Kseventhree \onesum \Kzero$}}%
\fffbox[\entrysize]{\vbox{\begin{tikzpicture}[scale=0.8,every edge/.style={draw,thick}]\input{tikz/r4.tikz}\end{tikzpicture} \\ $\Ksevenfour \onesum \Kzero$}}%
\\
\smallskip
\fffbox[\entrysize]{\vbox{\begin{tikzpicture}[scale=0.8,every edge/.style={draw,thick}]\input{tikz/r5.tikz}\end{tikzpicture} \\ $\Ksevenfive \onesum \Kzero$}}%
\fffbox[\entrysize]{\vbox{\begin{tikzpicture}[scale=0.8,every edge/.style={draw,thick}]\input{tikz/r6.tikz}\end{tikzpicture} \\ $\Ksevensix \onesum \Kzero$}}%
\fffbox[\entrysize]{\vbox{\begin{tikzpicture}[scale=0.8,every edge/.style={draw,thick}]\input{tikz/r7.tikz}\end{tikzpicture} \\ $\Ksevenseven \onesum \Kzero$}}%
\fffbox[\entrysize]{\vbox{\begin{tikzpicture}[scale=0.8,every edge/.style={draw,thick}]\input{tikz/r8.tikz}\end{tikzpicture} \\ $(\Kfour \twosum \Kthree) \onesum \Kzero$}}%
\\
\smallskip
\fffbox[\entrysize]{\vbox{\begin{tikzpicture}[scale=0.8,every edge/.style={draw,thick}]\input{tikz/r9.tikz}\end{tikzpicture} \\ $\Kfour \onesum \Kthree$}}%
\end{table}

\section{Preliminaries}\label{sec:prelim}
\subsection{Diagrams}  
Given a spatial graph $\Gamma$ and $p\in \sphere-\Gamma$, then we say a projection $\pi:\sphere-\{p\}\simeq \mathbb{R}^3\rightarrow \mathbb{R}^2$ is \emph{regular} with respect to $\Gamma$ if $\pi$ is in general position \cite{Zee:63} (see \cite[Section $4.4$]{Buo:03}). In particular, the image $\pi(\Gamma)$ 
is a plane graph with vertices the union of 
double points of $\pi$ and vertices of $\Gamma$. 

Given a regular projection $\pi$ with respect to $\Gamma$, a \emph{diagram} of $\Gamma$ is given by replacing each double point in $\pi(\Gamma)$ 
with a \emph{crossing}, where the broken arc corresponds 
to the edge of $\Gamma$ underneath the edge the unbroken arc corresponds to in $\mathbb{R}^3\simeq \sphere-\{p\}$; see Fig.\ \ref{fig:over_under_crossings}.
A \emph{diagram} of a handlebody-knot $V$ is a diagram of one of its spines.
The plane graph $\pi(\Gamma)$ is called \emph{the underlying graph} of the diagram.

The \emph{crossing number} of a diagram of a spatial graph $\Gamma$ (resp.\ handlebody-knot $V$) is the number of the crossings the diagram has. The \emph{crossing number} $c(\Gamma)$ of $\Gamma$ (resp.\ $c(V)$ or $V$) 
is the minimum of the crossing numbers over all diagrams of $\Gamma$ (resp.\ $V$). 
A diagram of $\Gamma$ (resp.\ $V$) that attains the crossing number of $\Gamma$ (resp.\ $V$) 
is called a \emph{minimal diagram}. A minimal diagram of $\Gamma$ is in general not unique.

Given a diagram $D$ of a spine $\Gamma$ of a handlebody-knot $V$, observe that if $v\in D$ 
is a vertex of valence greater than $3$ and 
$\rnbhd{v}\subset \mathbb{R}^2$ is a regular neighborhood of $v$ with $D_v:=D\cap \rnbhd{v}$ 
a deformation retract of $\rnbhd{v}$, 
then we can replace $D_v$ with a new graph 
$\tilde{D}_v$ where $\tilde D_v\cap \partial \rnbhd{v}=D_v\cap \partial \rnbhd{v}$
and all vertices of $\tilde D_v$ in the interior of $\rnbhd{v}$ 
are trivalent; see Fig.\ \ref{fig:turn_trivalent}.
In particular, every handlebody-knot has a minimal 
diagram which is a diagram of a trivalent spine---namely, all vertices are trivalent. 

\begin{figure}
\begin{subfigure}{.48\linewidth}
\centering
\includegraphics[scale=.12]{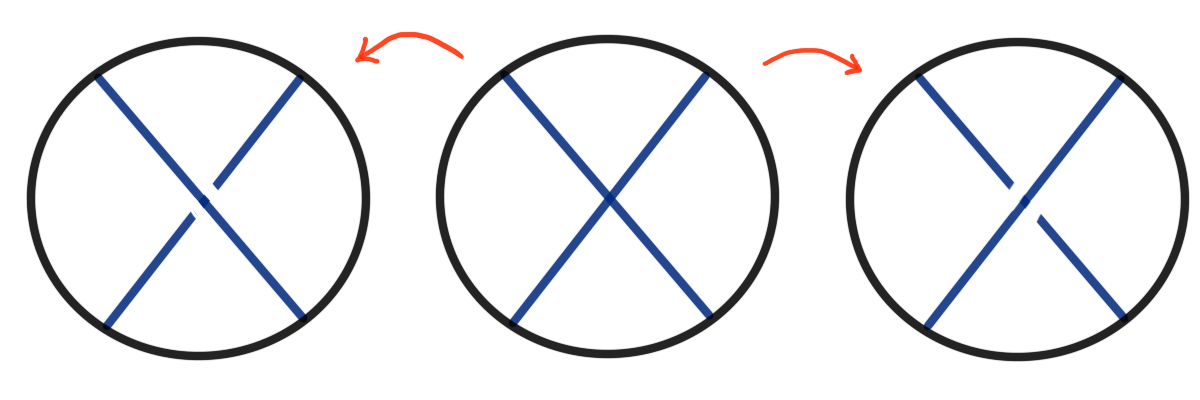} 
\caption{Turning $4$-valent vertices into crossings.}
\label{fig:over_under_crossings}
\end{subfigure}
\begin{subfigure}{.48\linewidth}
\centering
\includegraphics[scale=.12]{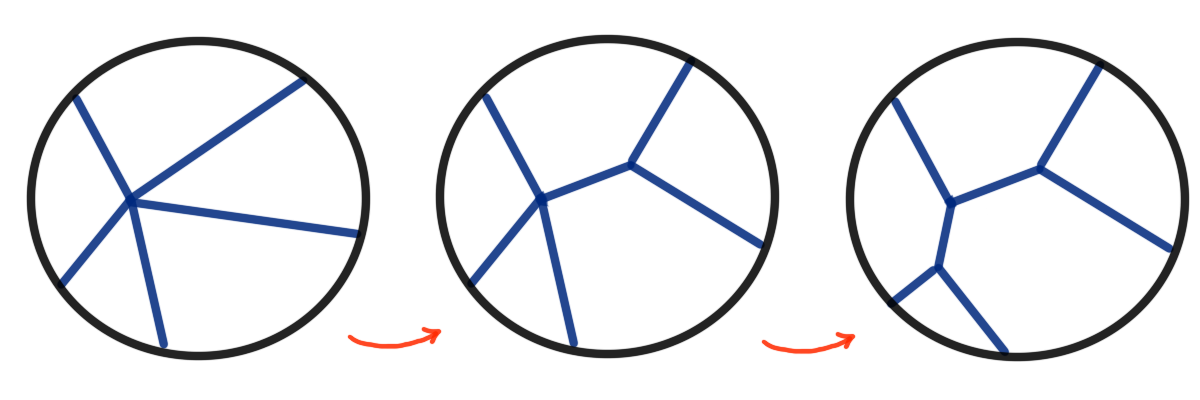}
\caption{Turning into trivalent vertices.}
\label{fig:turn_trivalent}
\end{subfigure} 
\caption{}
\end{figure}

A trivalent spine of a genus $g$ handlebody-knot 
has $2g-2$ trivalent vertices. Therefore, a trivalent spine of a genus two handlebody-knot is either a spatial theta graph or a spatial handcuff graph. For the sake of simplicity, hereinafter, a minimal diagram always refers to a minimal diagram of a trivalent spatial graph or a trivalent spine of a handlebody-knot.

Recall that the \emph{edge-connectivity} of a graph is the minimum number of edges whose deletion disconnects the graph. A graph is $e$-connected if it has edge-connectivity greater than or equal to $e$.
A diagram of a spatial graph is said to have \emph{connectivity $e$} (resp.\ to be \emph{$e$-connected}) if its underlying graph has edge-connectivity $e$ (resp.\ is $e$-connected). 
 
To study handlebody-knots in terms of their diagrams, we recall that \emph{generalized Reidemeister moves} are the local changes depicted in Figs.\ \ref{fig:reidemeister_a}
and \ref{fig:reidemeister_b} and an \emph{IH-move} is the local change in Fig.\ \ref{fig:IH-move}.

\begin{theorem}[\text{\cite[Corollary $2$]{Ish:08}}]\label{thm:IH_move_handlebody_knot}
Two handlebody-knots  
are equivalent if and only if their diagrams
are related by a finite sequence of 
generalized Reidemeister moves and IH-moves.
\end{theorem}
\begin{figure}[ht]
\includegraphics[height=0.18\textwidth]{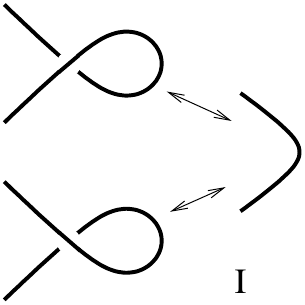}~~
\includegraphics[height=0.18\textwidth]{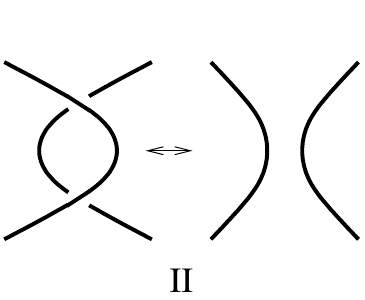}~~
\includegraphics[height=0.18\textwidth]{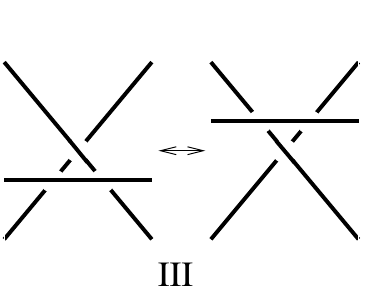}
\caption{Classical Reidemeister moves of type I, II, III.}
\label{fig:reidemeister_a}
\end{figure}
\begin{figure}[ht]
\includegraphics[height=0.18\textwidth]{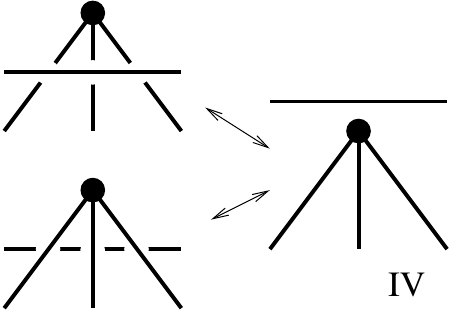}~~~~
\includegraphics[height=0.18\textwidth]{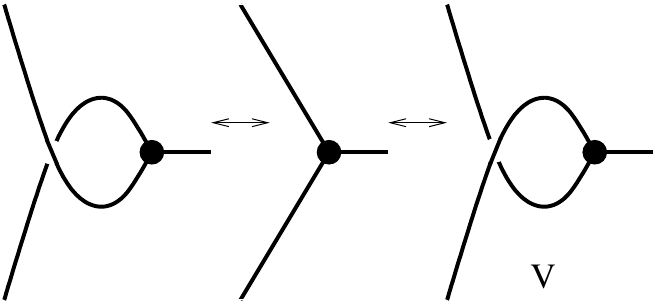}
\caption{Reidemeister moves IV and V involving a trivalent vertex.}
\label{fig:reidemeister_b}
\end{figure}
\begin{figure}[ht]
\includegraphics[height=0.18\textwidth]{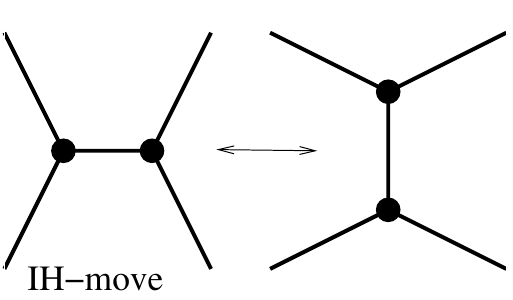}
\caption{IH-move.}
\label{fig:IH-move}
\end{figure}

\section{Completeness}\label{sec:completeness}

Hereinafter, for the sake of simplicity, unless otherwise specified, by a handlebody-knot, we understand a genus two handlebody-knot.
To enumerate all seven crossing handlebody-knots, our approach is to divide minimal diagrams into three classes: 
\begin{enumerate}
\item\label{itm:greater} $3$-connected minimal diagrams;
\item\label{itm:two} Minimal diagrams with connectivity $2$;
\item\label{itm:one} Minimal diagrams with connectivity $1$.   
\end{enumerate}
 
\subsection{$3$-connected diagrams}\label{sec:completeness3}
Our goal here is to find all seven crossing handlebody-knots that have a $3$-connected minimal diagram.
To do so, we first enumerate all $3$-connected plane graphs with two trivalent vertices and seven quadrivalent vertices,
and then replace each quadrivalent vertex with a crossing.%
\footnote{
\draftGB{we have to specify that the two notions
are different for us, it seems not always so in the literature}
Note that the Whitney theorem \cite{Whi:32} is not applicable here since a $3$-connected graph may not be $3$-vertex-connected.}
%
%
The first part is done by a code devised by the second author, and it results in $908$ such plane graphs, up to mirror image. 
Since, for each quadrivalent vertex, there are two possible choices of crossings (see Fig.\ \ref{fig:over_under_crossings}), these $908$ plane graphs produce an unmanageable number, $908\cdot 2^6$, of diagrams, up to mirror image. To reduce the number of diagrams, a software developed by the third author is implemented to drop diagrams that 
\begin{itemize}
\item{} represent a handlebody-link with more than one component, or
\item{} is recognized to be nonminimal via moves in Figs.\  \ref{fig:autononminimal} and \ref{fig:moreautononminimal}.
\end{itemize}
In addition, one choice is selected for 
any pair of over/under crossings
\begin{itemize}
	\item{} that results in embeddings having a portion depicted in Fig.\ \ref{fig:autopairs} or
	\item{} that are directly connected by an IH move (Fig.\ \ref{fig:IH-move}.
\end{itemize}

Consequently, we obtain a refined set $\mathcal D_3$ of %
\draftMP{A recent improvement in appcontour allows to reduce the previously reported number of $932$
to $910$, since now also the "barrel-fork" subdiagram is recognized as part of a nonminimal diagram
and then further reduced to $808$ by recognizing diagrams that are equivalent up to an IH move (only
one per pair is retained).
The process to arrive at the $808$ number consists in first computing the 'pairs7' list of embeddings
plus choices ($1632$ pairs), then the initial state of the ROOT of the bigtree is populated with the
resulting $1632$ embeddings. Finally the scripts "./search01x.sh" and ``./search{\_}uptoih.sh"  recognize
situations corresponding to the IH+V move of Figure \ref{fig:autononminimal} and pairs of
diagrams directly connected by an IH move. }%
$808$
diagrams, where, for every seven crossing handlebody-knot $V$ with a $3$-connected minimal diagram, there is a minimal diagram in $\mathcal{D}_3$ representing $V$.
We arrive at the final number of $808$ in two steps, first all situations depicted in Figures \ref{fig:autononminimal}
and \ref{fig:moreautononminimal} with the exception of the IH+V move are recognized by the \texttt{appcontour} software
and used to generate a set of $1632$ diagrams.
This set is then reduced to $808$ entries by automatically recognizing situations matching the move ``IH+V'' of Figure
\ref{fig:moreautononminimal} and, for any pair of diagrams that are directly connected by an IH move only one is retained.

\begin{figure}
	\includegraphics[height=0.3\textwidth]{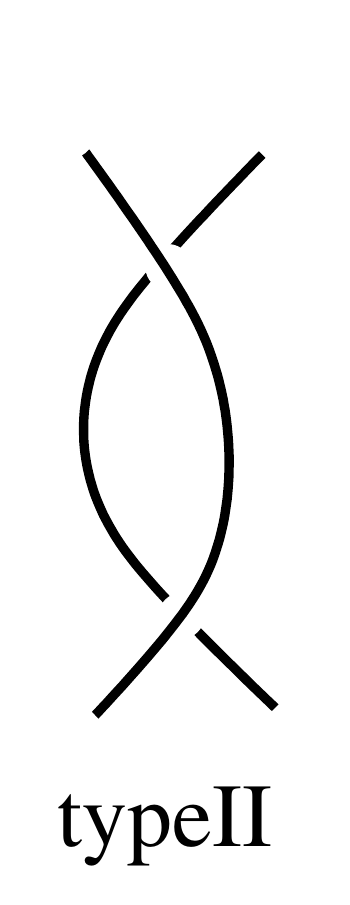}
	~~
	\includegraphics[height=0.3\textwidth]{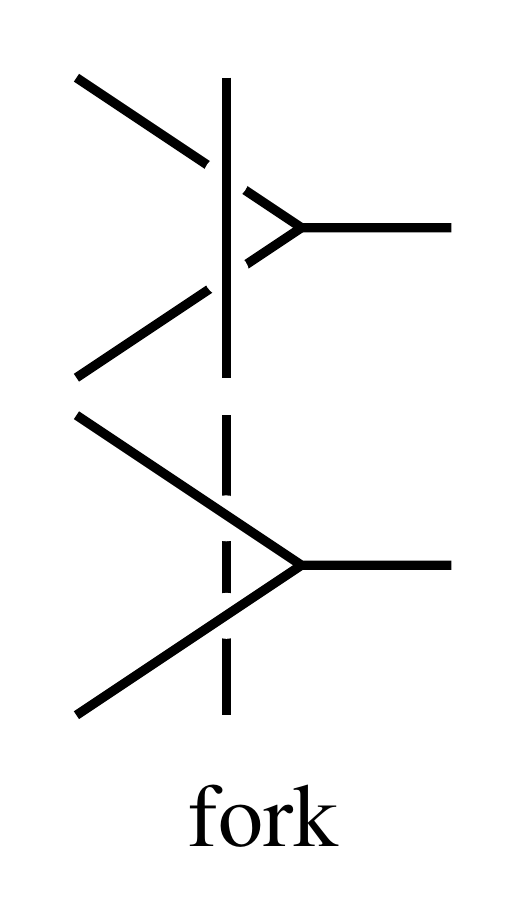}
	~~
	\includegraphics[height=0.3\textwidth]{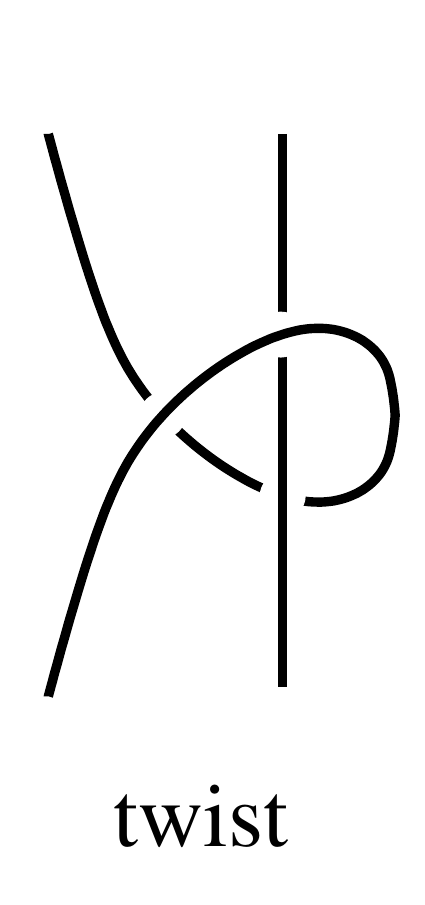}
	~~
	\includegraphics[height=0.3\textwidth]{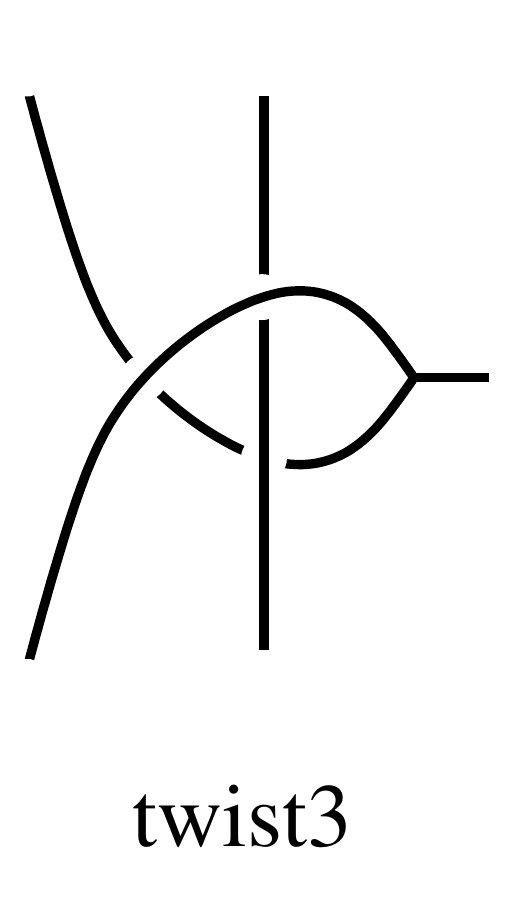}
	~~
	\includegraphics[height=0.3\textwidth]{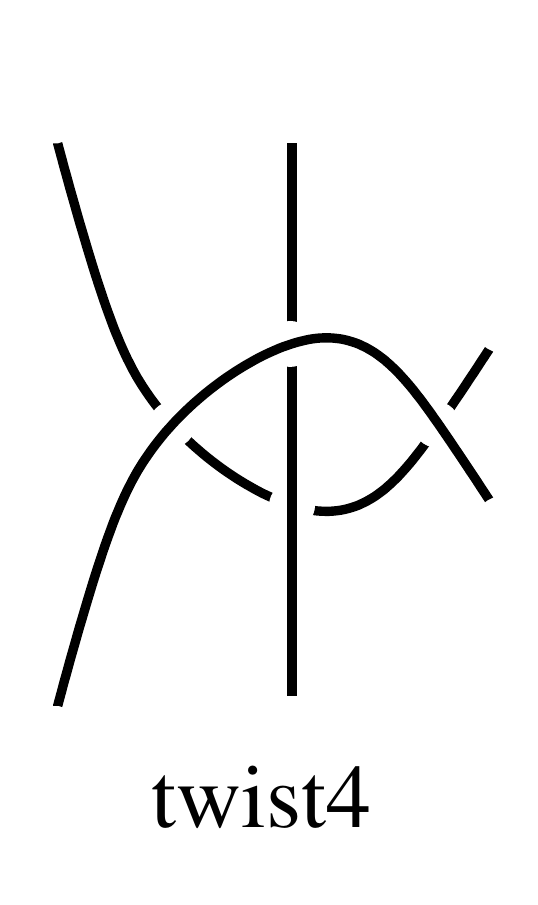}
	\caption{Crossing-reducing moves.  They consist of a sequence of Reidemeister moves
	as follows (from left to right):
	II~; II$^{-1}$ IV~; III I~; III V~; III II~. Notation ``II$^{-1}$" means that move II
	is performed in the ``crossing-increasing'' direction.}
	\label{fig:autononminimal}
\end{figure}

\begin{figure}
	\includegraphics[height=0.3\textwidth]{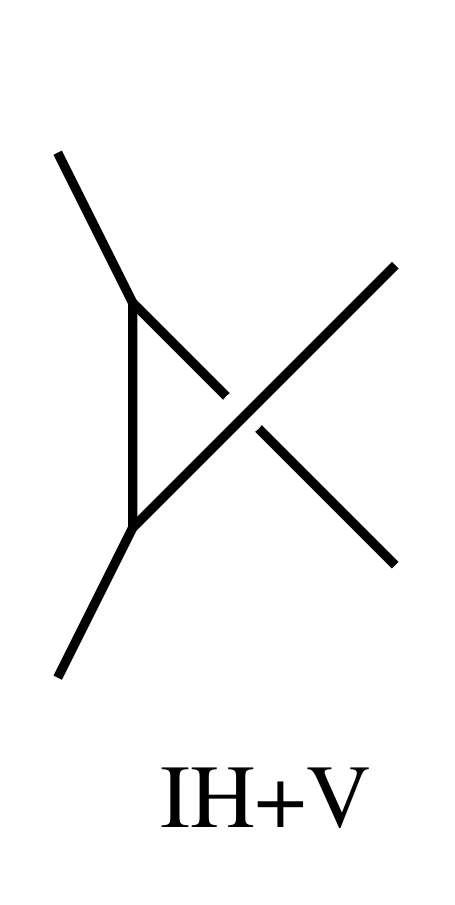}
	~~~~~~~
	\includegraphics[height=0.3\textwidth]{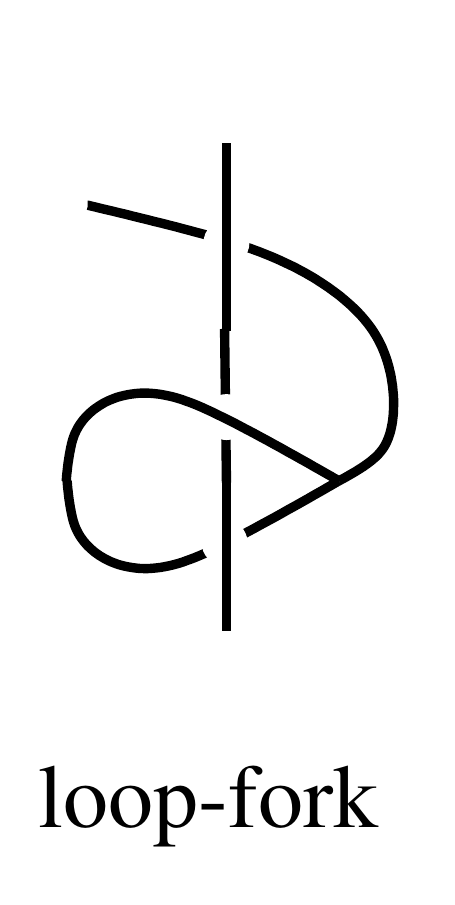}
	~~~~~~~
	\includegraphics[height=0.3\textwidth]{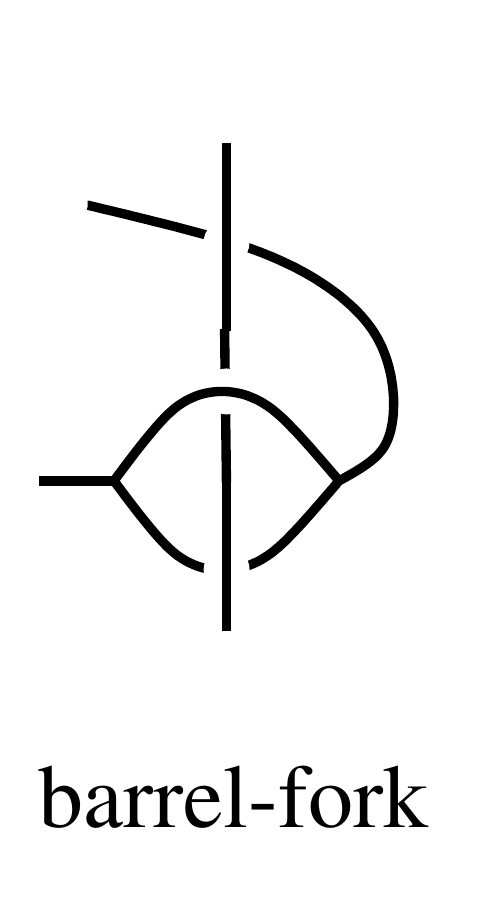}
	\caption{More crossing-reducing moves. ``IH+V'' is an IH-move followed by a V move (Fig. \ref{fig:reidemeister_b});
	``loop-fork'' and ``barrel-fork'' are performed by twisting the loop/barrel (see Fig.
	\ref{fig:autopairs}) and then applying the fork move
	(Fig.\ \ref{fig:autononminimal}).}
	\label{fig:moreautononminimal}
\end{figure}

\begin{figure}
	\includegraphics[height=0.3\textwidth]{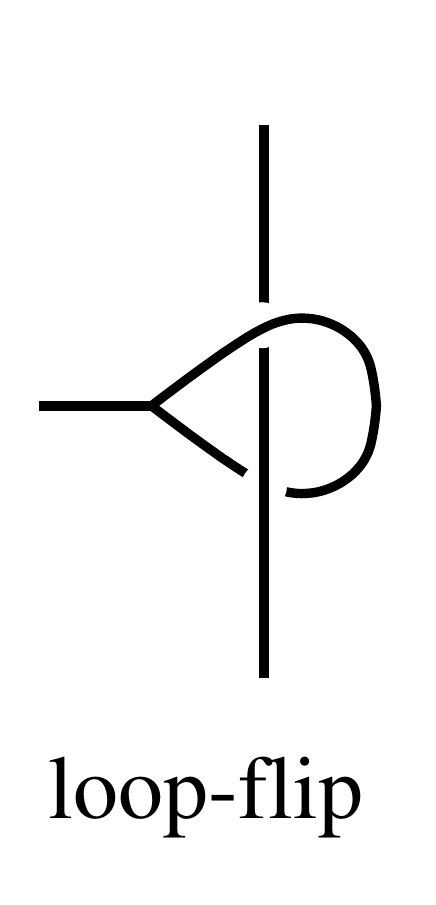}
	~~~~~~
	\includegraphics[height=0.3\textwidth]{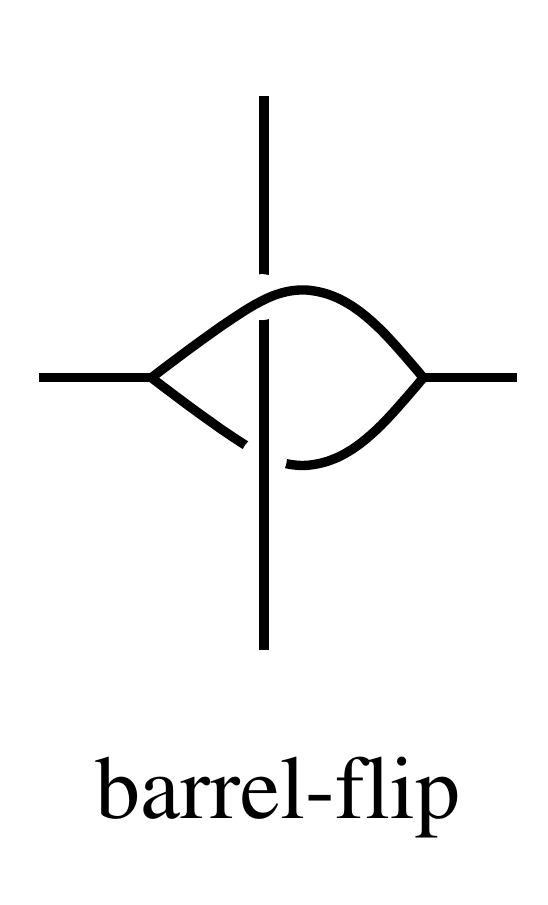}
	\caption{Loop/barrel flipping moves.  They consist of the following sequence of Reidemeister moves
	``loop-flip'': II$^{-1}$ V III I resulting in a reversal of the two overcrossings.
	``barrel-flip'': same sequence as for ``loop-flip'' with a V move replacing I,
	also resulting in a reversal of the overcrossings.}
	\label{fig:autopairs}
\end{figure}

The filtering process described by the three items above does not guarantee against nonminimal diagrams and
they must be manually uncovered during the {\it dynamic process} described in Section \ref{sec:bigtree}.
To this aim it greatly helps to include in the process also the handlebody-knots with fewer than
$7$ crossings, for which we have a complete list.
Indeed, a nonminimal $7$-crossings diagram must be indistinguishable by any invariant against some handlebody-knot
with fewer crossings raising the suspicion that it is actually equivalent to it.
If we are not able to manually prove equivalence, the
alternative is that it is actually minimal and we are dealing with a {\it hard-pair} for which we
must manually prove inequivalence.
Indeed we have two hard-pairs where one entry has fewer crossings (see Section \ref{sec:bigtree}).

Concerning the exclusion of handlebody-links we observe that the number of components of a diagram does not
depend on the way we chose the overpasses at crossings and can be easily computed in an automated way.

In \cite{supplement:26} we provide details for the sets described here, in particular the list of diagrams
in set $\mathcal{D}_3$, file \texttt{D3.csv}.

\subsection{Diagrams with connectivity $2$}
The aim here is to find all seven crossing handlebody-knots that have a minimal diagram with connectivity $2$ and no minimal diagram with connectivity $1$.

\label{sec:completeness2}
\subsection*{$2$-sum of diagrams and spatial graphs}
A $2$-sum of two spatial graphs $\Gamma_1,\Gamma_2$ 
is constructed as follows \cite{Mor:09a}: 
For each $i$, choose an edge $e_i$ in $\Gamma_i$; secondly, consider a $3$-ball $B_i$ with $B_i\cap\Gamma_i=B_i\cap e_i$; thirdly, orient $e_i$ so that $e_i$ is going out of $B_i$ at a point $a_i\in e_i\cap \partial B_i$ and going in $B_i$ at a point $b_i\in e_i\cap \partial B_i$.  
Then the \emph{order-$2$ vertex connected sum}, abbreviated to \emph{$2$-sum}, is the spatial graph given by gluing the exteriors
$\Compl{B_1}$ and $\Compl{B_2}$ with an orientation-reversing homeomorphism
$h:\partial \Compl{B_1}\rightarrow \partial \Compl{B_2}$ with $h(a_1)=b_2$, $h(b_1)=a_2$. The spatial graphs $\Gamma_1$, $\Gamma_2$ are called \emph{summands} of the $2$-sum.

Since a $2$-sum depends on the choice of edges in $\Gamma_i$ and a given orientation, it is not unique; for instance, Figs.\ \ref{fig:knot_edge}, \ref{fig:unknot_edge}, \ref{fig:connecting_edge} illustrate three different $2$-sums of an oriented knot $K$ and the spatial graph in Fig.\ \ref{fig:gamma}. If $K$ is not invertible, then two different orientations of the edge $e$ in Fig.\ \ref{fig:gamma} yield two different $2$-sums. 

\begin{figure}[t]
\begin{subfigure}{.23\linewidth}
\begin{overpic}[scale=.1,percent]{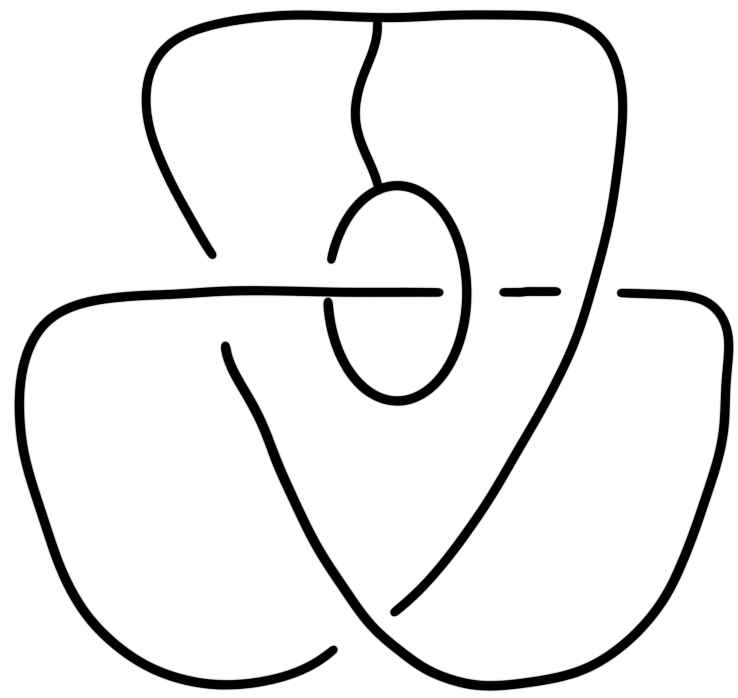} 
\put(53,77){$e$}
\end{overpic}
\caption{} 
\label{fig:gamma}
\end{subfigure}
\begin{subfigure}{.23\linewidth}
\begin{overpic}[scale=.1,percent]{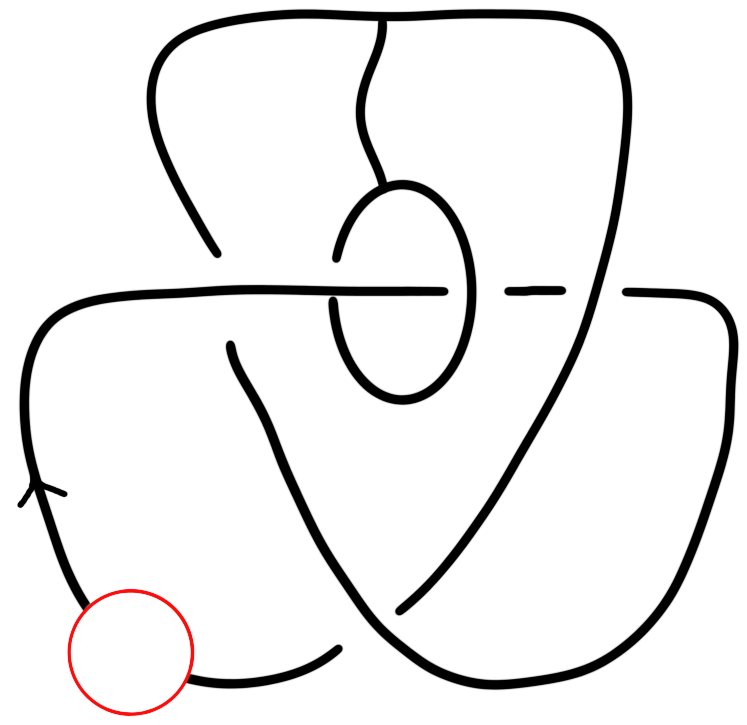} 
\put(11.8,4.8){$K$}	
\end{overpic}
\caption{} 
\label{fig:knot_edge} 
\end{subfigure}
\begin{subfigure}{.23\linewidth}
\begin{overpic}[scale=.1,percent]{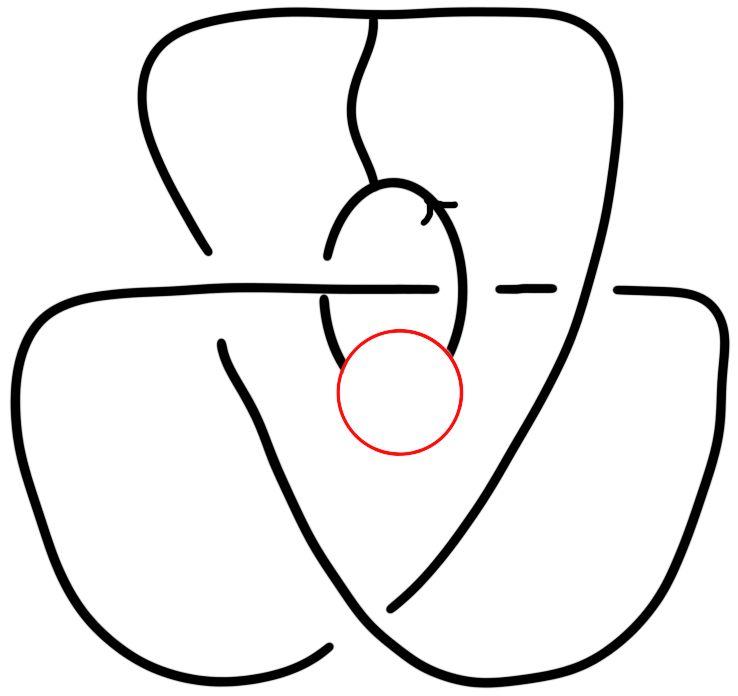} 
\put(48,36){$K$}		
\end{overpic}
\caption{} 
\label{fig:unknot_edge} 
\end{subfigure}
\begin{subfigure}{.23\linewidth}
\begin{overpic}[scale=.1,percent]{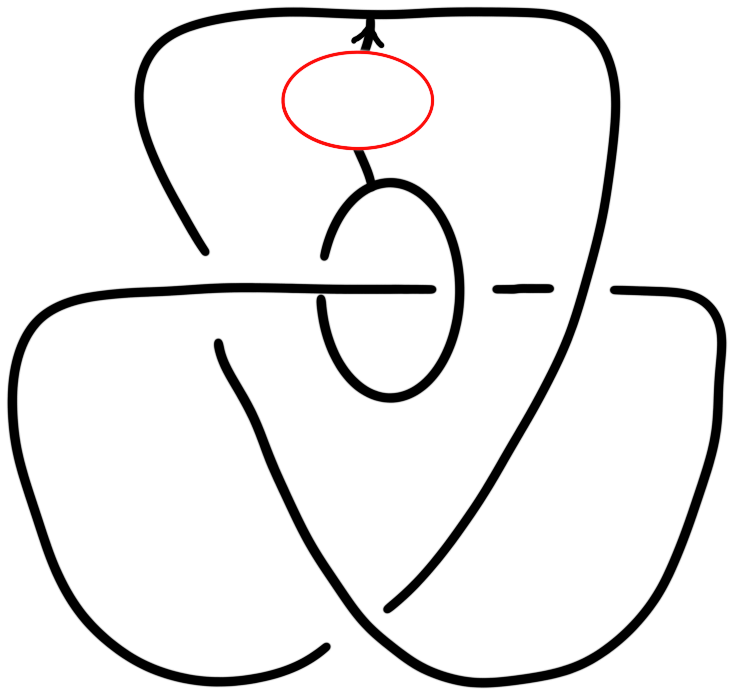}
\put(43,75.5){\small $K$}	
\end{overpic}
\caption{} 
\label{fig:connecting_edge} 
\end{subfigure}
\caption{$2$-sum.}
\label{fig:two_sum}
\end{figure}

We denote by $\Gamma_1\#\Gamma_2$ the set of all possible $2$-sums of $\Gamma_1$ or its mirror image and $\Gamma_2$ or its mirror image.

Given a \emph{diagram} $D$ of a handlebody-knot $V$, a prime factor $N\subset \mathbb{R}^2$ of $D$ is a disk $N$ with $\partial N$ meeting $D$ at two points $a,b$ such that $N\cap D$ together with an arc outside $N$ is a diagram of a prime knot $K$, called the \emph{induced prime knot}. A \emph{system} $\mathcal N$ of prime factors is a set of mutually disjoint prime factors, and it is \emph{maximal} if it is not a proper subset of another system of prime factors.  

Given a system $\mathcal{N}=\{N_1,\cdots, N_k\}$ of prime factors, let $\alpha_i$ be an arc in $N_i$ 
with $\partial \alpha_i = D\cap \partial N_i$, 
then $\big(D-(N_1\cup\cdots\cup N_k)\big)\cup (\alpha_1\cup\cdots\cup \alpha_k)$ is a diagram $D_0$ of a spatial handcuff or theta graph $\Gamma_0$; $D_0$ and $\Gamma_0$ are called the \emph{base diagram} and \emph{base graph}, respectively. The spatial graph $D$ represents can thus be regarded as a $2$-sum of $\Gamma_0$ and some prime knots.  

Suppose $D$ is minimal with $7$ crossings and connectivity $2$. Then $D$ admits a non-empty maximal system $\mathcal{N}=\{N_1,\cdots,N_k\}$ of prime factors, and 
the base graph $D_0$ is $3$-connected and minimal. Since $D$ has only $7$ crossings and a non-trivial knot has at least $3$ crossings, we have $k\leq 2$, namely,
\[\Gamma\in (\Gamma_0\# K_1) \text{ or } \Gamma\in (\Gamma_0\# K_1)\# K_2\]
where $K_1,K_2$ are the prime knots induced by
$N_1,N_2$, respectively. 
The base graph $\Gamma_0$ cannot be a trivial spatial handcuff since $V$ admits no minimal diagram with connectivity $1$. If $\Gamma_0$ is a trivial spatial $\theta$-graph, then, since $k\leq 2$, there is an edge in $D$ joining the two trivalent vertices. Applying
an IH-move along the edge turns $D$ into a diagram with connectivity $1$ without changing the number of crossings, contradicting our assumption. As a result, the base graph $\Gamma_0$ is non-trivial, and hence $D_0$ has at least two crossings. This implies $k=1$, and $K_1$ has at most five crossings, and $\Gamma_0$ has at most four crossings.

Since prime knots up to five crossings are all invertible, the set $\Gamma_0\# K_1$ contains at most six elements, depending on the arc in $\Gamma_0$ where the $2$-sum is performed and the chirality of $K_1$. Note that handlebody-knots so produced might not have seven crossings. 

To determine $\Gamma_0$, we note that, since 
$D_0$ is $3$-connected, it suffices to enumerate all spatial graphs that have a $3$-connected minimal diagram, up to four crossings. 
This is done by the code developed by the second author, which allows us to find all 
$3$-connected plane graphs with two trivalent vertices and up to four quadrivalent vertices. Such spatial graphs are enumerated in Table \ref{tab:graphs} (see \cite[Table $4$]{BePaPaWa:23} and \cite{Mor:09a,Mor:07b}).
Denote by $\op{Kn_m}$ the prime knot $\op{n_m}$ in the knot table. Then, since $\Gamma_0$ is in Table \ref{tab:graphs}, seven crossing handlebody-knots $V$ that have a minimal diagram with connectivity $2$ but no minimal diagram with connectivity $1$ can be enumerated as follows.


\begin{table}[t]
	\caption{Spatial graphs up to four crossings.
	}
	\label{tab:graphs}
	\includegraphics[height=0.15\textwidth]{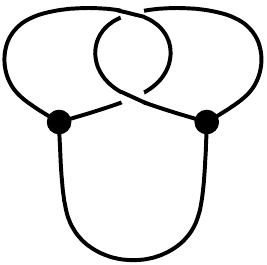}$\op{G2}$
	\includegraphics[height=0.15\textwidth]{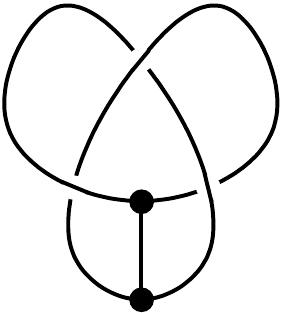}$\op{G3}$
	\includegraphics[height=0.15\textwidth]{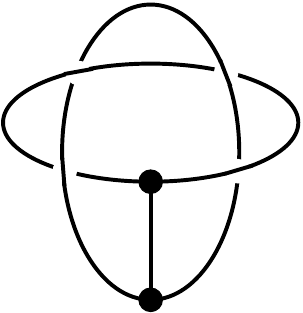}$\op{G4_1}$\\
	\vspace*{1.5em}
	\includegraphics[height=0.15\textwidth]{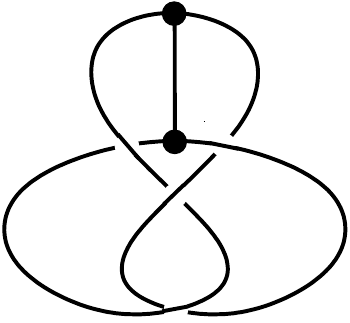}$\op{G4_2}$
	\includegraphics[height=0.15\textwidth]{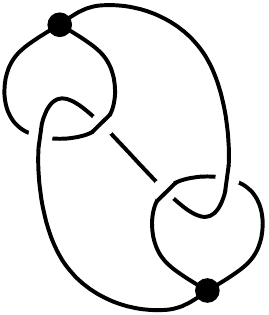}$\op{G4_3}$
\end{table}

\begin{itemize}
\item $\op{G2} \twosum \Kthree$: the handlebody-knot $5_4$ of \cite[Table 1]{IshKisMorSuz:12}; 
\item $\op{G3} \twosum \Kthree$: the handlebody-knots $6_{14}$ and $6_{15}$ of \cite[Table 1]{IshKisMorSuz:12}.
\item $\op{G2} \twosum \Kfour$: the handlebody-knot $6_{16}$ of \cite[Table 1]{IshKisMorSuz:12};
\item $\op{G2} \twosum \Kfiveone$: 
the handlebody-knot $\sevensixtyone$ of Table \ref{tab:table};
\item $\op{G2} \twosum \Kfivetwo$: 
the handlebody-knot $\sevensixtytwo$ of Table \ref{tab:table};
\item $\op{G3} \twosum \Kfour$: 
the handlebody-knot $\sevensixtythree$ of Table \ref{tab:table};
\item $\op{G4_1} \twosum \Kthree$: the handlebody-knots
$\sevensixtyfour$ and $\sevensixtyfive$ of Table \ref{tab:table}
\item $\op{G4_2} \twosum \Kthree$: the handlebody-knots 
$\sevensixtysix$ and $\sevensixtyseven$ of Table \ref{tab:table}.
\item $\op{G4_3} \twosum  \Kthree$: the handlebody-knots
$\sevensixtyeight$ and $\sevensixtynine$ of Table 
\ref{tab:table}.
\end{itemize}
Each handlebody-knot above admits a canonical diagram with a single prime factor associated to a minimal diagram of $\op{Kn_m}$; see Table \ref{tab:table}.
The set of diagrams is denoted by $\mathcal{D}_2$ and contains $14$ entries, listed in \cite{supplement:26}, file
\texttt{D2.csv}.
%
%
It is obtained as follows: first we construct all the embeddings (with no overpass information) matching the list above restricted to
$7$-crossings.
With the exception of $\op{G4_3}$ we must obstruct the IH move that would otherwise lead to nonminimality.
Since we lose, at this stage, the overcrossing information we can forget about the chirality of the factor, moreover due
to left-right simmetry we can also forget about the possible reversibility\footnote{We do not take advantage here of the actual
reversibility of the knots involved.} of the knot summand leading to only one (flattened) embedding.
On the contrary $\op{G4_3} \twosum  \Kthree$ contributes four distinct (flattened) embeddings, they are distinguished by the choice of
the arc in $\op{G4_3}$, two possibilities since we can exchange the two loops, then we have two choices for the glueing of the end-points
in the surgery (that could be reduced to one by enforcing the reversibility of $\Kthree$).
This set of $9$ embeddings will then produce the $14$ pairings with a choice of the overpass information after dropping trivially nonminimal
choices.

\subsection{Diagrams with connectivity $1$}\label{sec:completeness1}
The aim here is to find all seven crossing handlebody-knots that have a minimal diagram with connectivity $1$.

\subsection*{$1$-sum of handlebody-knots} 
Let $V_1,V_2$ be two handlebody-knots of genus $g_1,g_2$, respectively. Then the $1$-sum $V_1\onesum V_2$ is constructed as follows:
For each $i=1,2$, choose a $3$-ball $B_i$
such that $D_i:=V_i\cap B_i=\partial V_i\cap \partial B_i$ is a disk. 
Let $h:\partial \Compl{B_1}\rightarrow \partial \Compl{B_2}$ be an orientation-reversing homeomorphism with 
$h(D_1)=D_2$. 
Then $V_1\onesum V_2$ is the genus $g_1+g_2$ handlebody-knot obtained by gluing $D_1\subset V_1$ to $D_2\subset V_2$ via $h$ in the $3$-sphere $\Compl{B_1}\cup_h \Compl{B_2}$; see Fig.\ \ref{fig:one_sum}.

\begin{figure}[h]
	\centering
	\begin{overpic}[scale=.2,percent]{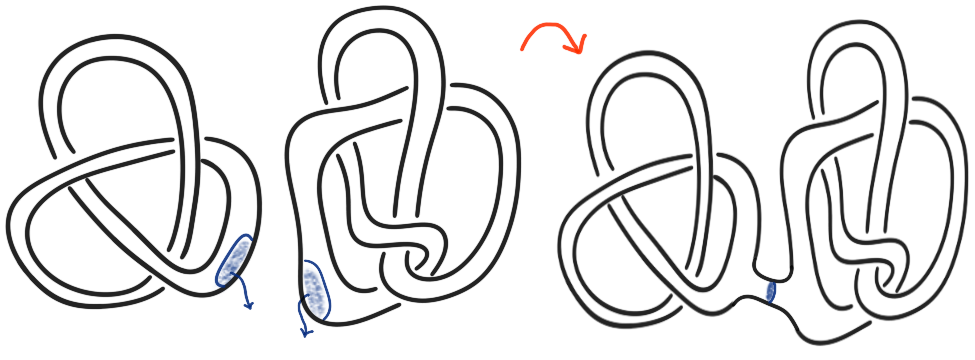}
		\put(29,-1){\footnotesize $D_2$}
		\put(23.2,1.8){\footnotesize $D_1$} 
	\end{overpic}
	\caption{One sum of $\op{K3_1}$ and $\op{K4_1}$.}
	\label{fig:one_sum}
\end{figure}

\subsection*{Enumeration}
Suppose the handlebody-knot $V$ has a minimal diagram $D$ with connectivity $1$, and let $e\in D$ be the edge cutting which disconnects $D$. 
Then $D-\{e\}$ consists of two components $D_1,D_2$.
Let $V_1,V_2$ be the genus one handlebody-knots represented by 
$D_1,D_2$. Then $D$ is a diagram of $V_1\onesum V_2$. 
Therefore, to enumerate handlebody-knots with a minimal diagram of connectivity $1$, it suffices to consider all one sums $V_1\onesum V_2$ with 
$V_1,V_2$ of genus one and $c(V_1)+c(V_2)=7$.
Since a genus one handlebody-knot is equivalent to a knot, we use the same notation $\op{Kn_m}$ to denote the genus one handlebody-knot corresponding to the prime knot $\op{Kn_m}$, and let $\Kzero$ be the trivial genus one handlebody-knot.

\begin{itemize}
	\item{} $\Ksevenone \onesum \Kzero$;
	\item{} $\Kseventwo \onesum \Kzero$;
	\item{} $\Kseventhree \onesum \Kzero$;
	\item{} $\Ksevenfour \onesum \Kzero$;
	\item{} $\Ksevenfive \onesum \Kzero$;
	\item{} $\Ksevensix \onesum \Kzero$;
	\item{} $\Ksevenseven \onesum \Kzero$;
	\item{} $\Kfour \onesum \Kthree$;
	\item{} $(\Kfour \twosum \Kthree) \onesum \Kzero$.
\end{itemize}
For each handlebody-knot above, there is a canonical diagram given by minimal diagrams of $\op{Kn_m}$ as shown in Table \ref{tab:tablered}; the set of diagrams is denoted by $\mathcal{D}_1$.

\subsection{Conclusion}
Let $\mathcal{D}=\mathcal{D}_1\cup\mathcal{D}_2\cup\mathcal{D}_3$. 

\begin{theorem}[Completeness]
For every seven crossing handlebody-knot $V$, 
there is a minimal diagram in $\mathcal{D}$ representing $V$. 
\end{theorem}
\begin{proof} 
As shown in Section \ref{sec:completeness1}, if $V$ has a minimal diagram with connectivity $1$, then $\mathcal{D}_1$ contains a minimal diagram of $V$. 
Suppose $V$ has no minimal diagram with connectivity $1$. If $V$ has a minimal diagram with connectivity $2$, Section \ref{sec:completeness2} implies $\mathcal{D}_2$ contains a minimal diagram of $V$; otherwise 
$V$ has a minimal diagram in $\mathcal{D}_3$; see Section \ref{sec:completeness3}. 
\end{proof}

\begin{remark}
In general, it is not known whether the minimal diagram of a $1$-sum always has connectivity $1$.
Similarly, it is unclear whether a minimal diagram of a handlebody-knot having a spine that is a $2$-sum always has connectivity $2$.

The enumeration is, however, based on the connectivity of minimal diagrams, and not on the decomposition type of handlebody-knots, and therefore does not depend on the aforementioned open problems. For instance, if a $1$-sum has no minimal diagrams with connectivity $1$, then its minimal diagram(s) can be found in $\mathcal{D}_2\cup \mathcal{D}_3$. 
\end{remark}

 \begin{remark}
 	Recall that a spatial graph $\Gamma$ is \emph{prime} if, for any $3$-ball $B$ with $\partial B$ meeting $\Gamma$ transversally at less than or equal to $3$ points, the intersection $B\cap \Gamma$ is contained in a proper disk in $B$; see \cite{Mor:07b}, \cite{Mor:09a}. 
 	In particular, if $\Gamma$ has a diagram  with connectivity less than or equal to $3$, then $\Gamma$ cannot be prime. 
 	Tables of prime spatial handcuff graphs and $\theta$-graphs, up to seven crossings, are given in Moriuchi \cite{Mor:07b}, \cite{Mor:09a}. 
    Denote by $\mathcal{G}$ be the set of Moriuchi's spatial graphs, up to equivalence and mirror image; there are $134$ in total, and let $\mathcal{H}$ be the set of genus two handlebody-knots with up to seven crossings, up to equivalence and mirror image; there are $90$ in total. Consider the function \[\mathcal{N}:\mathcal{G}\rightarrow \mathcal{H}\] 	
	given by taking a regular neighborhood of a spatial graph (see Table \ref{tab:map_graphs}).
	Then a natural question arises as to what the image of $\mathcal{N}$ is. Applying the IH-move shows that among all spatial $\theta$-graph with seven crossings in \cite{Mor:09a}, from $7_{25}$ onward, all except $7_{40},7_{48},7_{58},7_{60}, 7_{64}$, are non-minimal as a diagram of a handlebody-knot.
 	Similarly, among all spatial handcuff graph with seven crossings in \cite{Mor:07b}, from $7_{18}$ onward, all except $7_{21}$, are non-minimal as a diagram of a handlebody-knot. On the other hand, those who are minimal, as a diagram of a handlebody-knot, often represent equivalent handlebody-knots, up to mirror image. More precisely, 
 	\begin{itemize}
	\item spatial $\theta$-graphs $7_{21},7_{60}$ in \cite{Mor:09a} and spatial handcuff graph $7_{21}$ in \cite{Mor:07b} represent the handlebody-knot $7_{36}$ in Table \ref{tab:table};
 	\item spatial $\theta$-graphs $7_{18},7_{48},7_{58}$ in \cite{Mor:09a} represent the handlebody-knot $7_{52}$ in Table \ref{tab:table};
 	\item spatial $\theta$-graph $7_{16}$ in \cite{Mor:09a} and spatial handcuff graph $7_{17}$ in \cite{Mor:07b} represent the handlebody-knot $7_{32}$ in Table \ref{tab:table}. 	
 	\end{itemize}
 	In summary, there are $27$ handlebody-knots of seven crossings in Table \ref{tab:table} that are not in the image of $\mathcal{N}$, namely, all handlebody-knots from $7_{40}$ onward except $7_{47},7_{52}$. Particularly, none of their seven-crossing spines is prime.


 \end{remark}

\begin{table}[ht]
\caption{Map $\mathcal N$.
We omit graphs that are mapped to the trivial $\Kzero \onesum \Kzero$ handlebody knot.
The numbering for $\theta$-graphs (first table) and handcuff graphs (second table) refers to \cite[Fig. 1]{Mor:09a}
and \cite[Fig. 1]{Mor:07b} respectively.
The numbering for handlebodies (`h.body' columns) is taken from Table \ref{tab:table} and \cite[Table 1]{IshKisMorSuz:12}}.
\label{tab:map_graphs}
\begin{tabular}{ |c c|c c|c c| }
\hline
{\bf $\theta$-graph} & {\bf h.body} &
{\bf $\theta$-graph} & {\bf h.body} &
{\bf $\theta$-graph} & {\bf h.body}
 \\
\hline
$\op{5_1}$ & $5_3$ &
$7_{10}$ & $7_{26}$ &
$7_{35}$ & $5_{2}$
\\
$\op{5_2}$ & $5_2$ &
$7_{11}$ & $7_{27}$ &
$7_{39}$ & $5_{2}$
\\
$\op{6_1}$ & $6_8$ &
$7_{12}$ & $7_{28}$ &
$7_{40}$ & $7_{47}$
\\
$\op{6_2}$ & $6_9$ &
$7_{13}$ & $7_{29}$ &
$7_{41}$ & $6_{12}$
\\
$\op{6_3}$ & $6_{12}$ &
$7_{14}$ & $7_{30}$ &
$7_{45}$ & $6_{13}$
\\
$\op{6_4}$ & $6_{13}$ &
$7_{15}$ & $7_{31}$ &
$7_{47}$ & $6_{6}$
\\
$\op{6_7}$ & $4_1$ &
$7_{16}$ & $7_{32}$ &
$7_{48}$ & $7_{52}$
\\
$\op{6_{11}}$ & $5_1$ &
$7_{17}$ & $7_{33}$ &
$7_{51}$ & $6_{11}$
\\
$\op{6_{15}}$ & $5_1$ &
$7_{18}$ & $7_{52}$ &
$7_{52}$ & $6_{8}$
\\
$\op{7_1}$ & $7_{17}$ &
$7_{19}$ & $7_{34}$ &
$7_{54}$ & $6_{2}$
\\
$\op{7_2}$ & $7_{18}$ &
$7_{20}$ & $7_{35}$ &
$7_{55}$ & $6_{5 }$
\\
$\op{7_3}$ & $7_{19}$ &
$7_{21}$ & $7_{36}$ &
$7_{58}$ & $7_{52}$
\\
$\op{7_4}$ & $7_{20}$ &
$7_{22}$ & $7_{37}$ &
$7_{60}$ & $7_{36}$
\\
$\op{7_5}$ & $7_{21}$ &
$7_{23}$ & $7_{38}$ &
$7_{61}$ & $6_{9 }$
\\
$\op{7_6}$ & $7_{22}$ &
$7_{24}$ & $7_{39}$ &
$7_{62}$ & $5_{2 }$
\\
$\op{7_7}$ & $7_{23}$ &
$7_{26}$ & $6_{13}$ &
$7_{64}$ & $7_{47}$
\\
$\op{7_8}$ & $7_{24}$ &
$7_{30}$ & $6_{2}$ &
 &
\\
$\op{7_9}$ & $7_{25}$ &
$7_{31}$ & $5_{2}$ &
 & 
\\
\hline
\end{tabular}
%
%
%
\begin{tabular}{ |c c|c c|c c| }
\hline
{\bf h. graph} & {\bf h.body} &
{\bf h. graph} & {\bf h.body} &
{\bf h. graph} & {\bf h.body}
 \\
\hline
$\op{6_1}$ & $6_7$ &
$\op{7_{6 }}$ & $7_{6 }$ &
$\op{7_{17}}$ & $7_{32}$
\\
$\op{6_{2 }}$ & $6_{4 }$ &
$\op{7_{7 }}$ & $7_{7 }$ &
$\op{7_{20}}$ & $5_{3 }$
\\
$\op{6_{3 }}$ & $6_{3 }$ &
$\op{7_{8 }}$ & $7_{8 }$ &
$\op{7_{21}}$ & $7_{36}$
\\
$\op{6_{4 }}$ & $6_{2 }$ &
$\op{7_{9 }}$ & $7_{9 }$ &
$\op{7_{24}}$ & $4_{1 }$
\\
$\op{6_{6 }}$ & $5_{2 }$ &
$\op{7_{10}}$ & $7_{10}$ &
$\op{7_{25}}$ & $6_{8 }$
\\
$\op{6_{9 }}$ & $5_{2 }$ &
$\op{7_{11}}$ & $7_{11}$ &
$\op{7_{27}}$ & $5_{3 }$
\\
$\op{7_{1 }}$ & $7_{1 }$ &
$\op{7_{12}}$ & $7_{12}$ &
$\op{7_{29}}$ & $6_{5 }$
\\
$\op{7_{2 }}$ & $7_{2 }$ &
$\op{7_{13}}$ & $7_{13}$ &
$\op{7_{31}}$ & $6_{12}$
\\
$\op{7_{3 }}$ & $7_{3 }$ &
$\op{7_{14}}$ & $7_{14}$ &
$\op{7_{32}}$ & $6_{6 }$
\\
$\op{7_{4 }}$ & $7_{4 }$ &
$\op{7_{15}}$ & $7_{15}$ &
$\op{7_{33}}$ & $5_{1 }$
\\
$\op{7_{5 }}$ & $7_{5 }$ &
$\op{7_{16}}$ & $7_{16}$ &
$\op{7_{34}}$ & $6_{9 }$
\\
\hline
\end{tabular}
\end{table}


 

\section{No duplicates in the table}\label{sec:uniqueness}
Our goal is to classify diagrams in $\mathcal{D}$ into several classes, each of which contains only diagrams that represent the same handlebody-knot, up to equivalence and mirror image. 
Given the sheer size of $\mathcal{D}$, 
we adopt a dynamic approach to classify diagrams in $\mathcal{D}$.  

\subsection{A dynamic approach}
\label{sec:bigtree}
The dynamic approach can be pictured as a growing rooted tree. The end result is a big rooted tree $\mathcal{T}$, which is a dynamic data structure with data attached to each leaf. 
Specifically, 
each leaf of $\mathcal{T}$ contains diagrams that represent one or two equivalence classes of handlebody-knots, up to mirror image.

The tree $\mathcal{T}$ is grown from the
degenerate tree with a single vertex, the root, containing all diagrams in $\mathcal{D}$. 
At start, we apply some weak but computationally efficient invariant $t_1$ to handlebody-knots represented by diagrams in the root. The diagrams in $\mathcal{D}$ are then divided into classes so that each class contains only diagrams with the same $t_1$-invariant. We associate each class with a new leaf connecting to root with an edge, and thus the initial degenerate tree grows into a star.
Next, for each leaf $v$ of the star, we apply a slightly stronger invariant $t_2$, specifically tailored to the leaf $v$.
Thus $t_2$ divides the diagrams in the leaf $v$ into even smaller classes based on their $t_2$-invariants.
In the same way, we associate each class with a new leaf connecting to $v$ with an edge; note that $v$ is no longer a leaf at this stage.  
The process repeats itself for each leaf, and whenever we arrive at a stage where the leaf contains only a handful of diagrams or no computationally affordable invariant can distinguish the handlebody-knots the diagrams therein represent, we tag
the leaf as {\it final}.
For each final leaf we need to manually check equivalence of all entries therein.  Entries with a higher number of crossings must
be manually checked to be nonminimal.
In a few cases, discussed below, we arrive to a partition of the diagrams 
into two sets of equivalent diagrams that must \draftGB{[maybe not completely 
clear to say first that the diagrams are equivalent and then that are inequivalent]}
then be
manually shown to be inequivalent (hard pairs).
The big tree $\mathcal{T}$ is then the resulting tree, where every leaf is tagged as final; see the web page \cite{bigtable:web} for a snapshot.

The invariants used here are the Kitano-Suzuki invariants. 
Given a handlebody-knot $V$, let $G$ be a finite group; \emph{the Kitano-Suzuki $G$-invariant $ks_G$} is the number of conjugate classes of homomorphisms from the fundamental group of the exterior $\Compl V$ of $V$ to $G$; see \cite{KitSuz:12}.
When the order of $G$ is small, the computation load of $ks_G$ is rather light, but it grows exponentially as the order of $G$ gets larger. 
The computational complexity is also highly sensitive to the number of generators in the used presentation of the fundamental group. The finite groups used in our computation are: 
	\begin{itemize}
		\item{}$\operatorname{SL}_2(\Z_p)$ with prime $p$,
		\item{}The symmetric group $S_n$, $n \in {\mathbb N}$,
		\item{}The alternating group $A_n$, $n \in {\mathbb N}$.
	\end{itemize}
In the process of building the tree $\mathcal{T}$, we start with $G$ of a small size and move on to larger ones. 

The invariant $ks_G$ depends only on the exterior of a handlebody-knot, so it cannot differentiate inequivalent handlebody-knots with homeomorphic exteriors; such handlebody-knots, however, abound \cite{Mot:90}, \cite{LeeLee:12}, \cite{BePaWa:20a}. A stronger invariant, the {\it $G$-image} invariant introduced in \cite{BePaWa:20a} is thus used. 
The {\it $G$-image} invariant is constructed as follows: Consider the homomorphisms
\[ 
{\iota_1}_* :
	\pi_1(\partial V)\rightarrow \pi_1(V),\quad
	{\iota_2}_* :
	\pi_1(\partial V)\rightarrow \pi_1(\Compl V)
\]	
induced by the inclusions $\iota_1 : \partial V \to V$ and $\iota_2 : \partial V \to \Compl V$.
Then a surjective homomorphism from $\pi_1(\Compl V)$ to a finite group $G$ is said to be \textit{proper} if it becomes non-surjective after precomposed with ${i_2}_*$. The $G$-image is then given by the set of the images of the subgroup ${i_2}_\ast(\operatorname{Ker}({i_1}_\ast))$ in $G$ under proper homomorphisms.

The computation of the invariants listed above are automated by  
a software developed by the third author.
Once we obtain the tree $\mathcal{T}$, 
we check semi-manually, via Theorem \ref{thm:IH_move_handlebody_knot}, 
the equivalence of diagrams in each leaf.
Most of them are shown to contain only one equivalence class of handlebody-knots, up to mirror image.
Details of such equivalence proofs are contained in the supplementary material \cite{supplement:26} where we include both visual proofs via handwritten drawings and
certifying files that
list the sequence of Reidemeister plus IH moves that connect pairs of diagrams,
obtained with the aid of GPT-5.5.

As an exception, there are seven leaves that are shown to contain two equivalence classes.
These pairs are
$(5_1,\sevenfiftytwo)$,  
$(\sevenforty, \sevenfortyone)$, $(\sevenfortythree, \sevenfortyfour)$, $(\sevenfiftynine$, $\sevensixty)$, $(6_{12},\seventhirtynine)$, $(\sevensixtyfour,\sevensixtyfive)$, $(\sevensixtysix, \sevensixtyseven)$, each leaf containing handlebody-knots that have the same values of all computed invariants.
Here $5_1,6_{12}$ are taken from the table \cite{IshKisMorSuz:12}.
Other methods are needed to distinguish them.
The situation with the last two pairs $(\sevensixtyfour, \sevensixtyfive)$, $(\sevensixtysix, \sevensixtyseven)$ is relatively simple since handlebody-knots in each of them have spines that are $2$-sums with a right and left trefoil knots. By the uniqueness of knotted handle decomposition in \cite[Theorem $2.2$]{IshKisOza:15},
if $\sevensixtyfour, \sevensixtyfive$ 
(resp.\ $\sevensixtysix, \sevensixtyseven$) 
were equivalent, up to mirror image, then an orientation-reversing self-homeomorphism of $\sphere$ would exist preserving 
a tunnel of a link $L$, where $L$ is the four-crossing link L$4a1$ (resp.\ the figure eight K$4_1$), contradicting
\cite[Theorem $16.2$]{ChoMcC:09}; see Fig.\ \ref{fig:L4a1} (resp.\ Fig.\ \ref{fig:K41}).
\begin{figure}[h]
	\begin{subfigure}{.45\linewidth}
		\centering
		\begin{overpic}[scale=.12,percent]{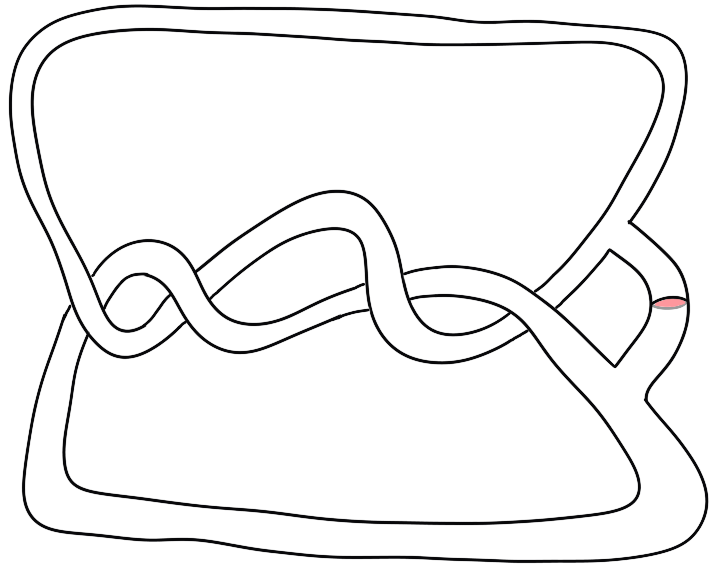} 
			\put(97,32){\small $D$}
		\end{overpic}
		\caption{L$4a1$ with the tunnel dual to $D$.} 
		\label{fig:L4a1}
	\end{subfigure}
	\begin{subfigure}{.45\linewidth}
		\centering
		\begin{overpic}[scale=.11,percent]{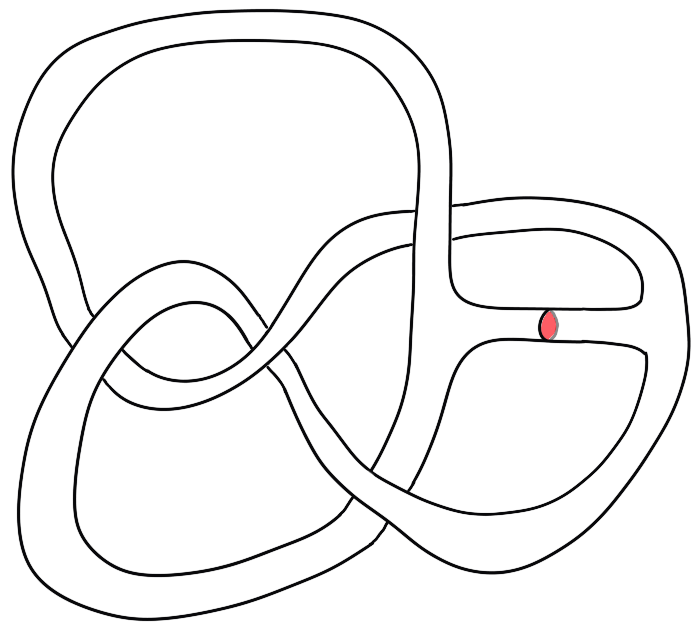} 
			\put(75,31){\small $D$}	
		\end{overpic}
		\caption{K$4_1$ with a tunnel dual to $D$.} 
		\label{fig:K41} 
	\end{subfigure}
	\caption{Neighborhood of the union of $L$ and its tunnel.}
\end{figure} 
On the other hand, \cite[Theorem $1.3$]{BePaPaWa:25} classifies handlebody-knots obtained from tangle replacement on the spatial graph $\mathrm{G}4_1$ in Fig.\ \ref{tab:graphs} (\cite{BePaPaWa:25} uses $\mathbf{h4_1}$ to denote $\mathrm{G}4_1$), and proves $6_{12}, \seventhirtynine$, $\sevenfiftynine, \sevensixty$ are mutually inequivalent, up to mirror image; see \cite[Corollary $1.4$]{BePaPaWa:25} for the statement and Figs.\ \ref{fig:replacement_Gfourone} for an illustration.
The inequivalence, up to mirror image, for members in the other three pairs: $(\sevenforty, \sevenfortyone)$, $(5_1,\sevenfiftytwo)$, and $(\sevenfortythree, \sevenfortyfour)$,
is proved in Section \ref{sec:hardpairs}.   
The end result is the union of Tables \ref{tab:table} and \ref{tab:tablered}.
\footnote{The authors are informed that the pairs $(7_{40},7_{41})$, $(7_{43},7_{44})$, $(7_{59},7_{60})$ can be differentiated by quandle invariants, and Ishii-Kishimoto have independently obtained a similar table with $69$ entries, where they find $3$ hard pairs.} 
\begin{figure}[h]
	\centering
\includegraphics[scale=.1]{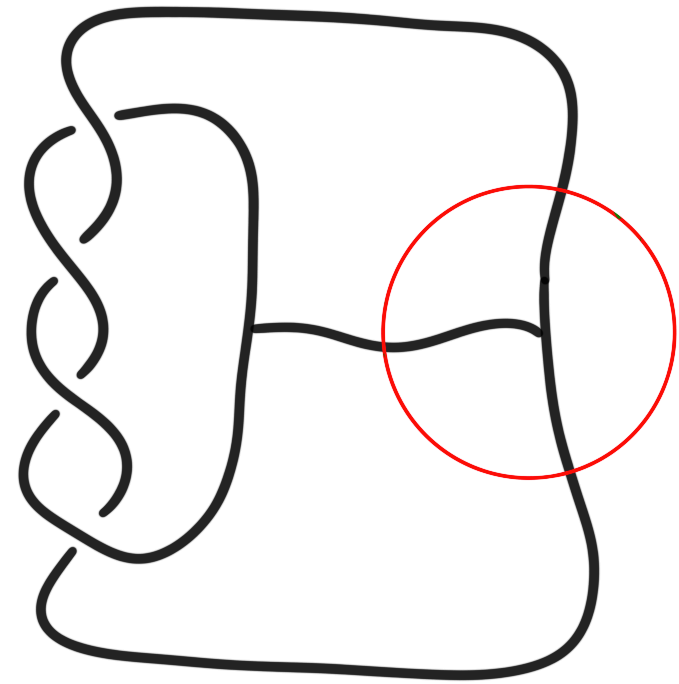}
$\mathrm{G}4_1$
\quad
\includegraphics[scale=.1]{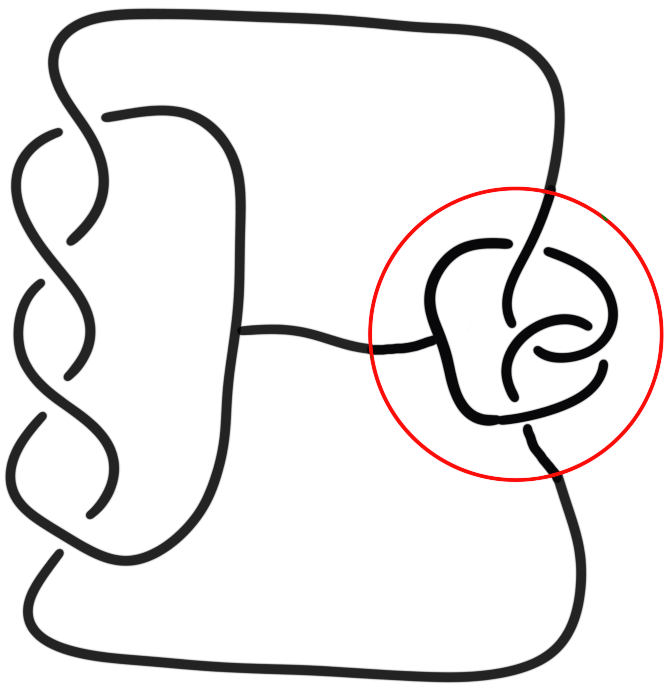}
$6_{12}$ 
\quad 
\includegraphics[scale=.1]{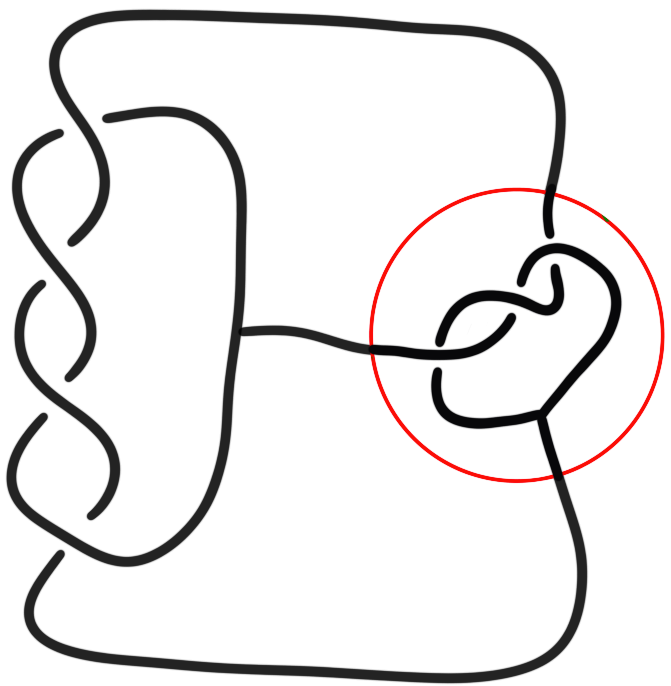}
$7_{59}$ 
\caption{Tangle replacement on $\mathrm{G}4_1$.}
	\label{fig:replacement_Gfourone}
\end{figure}
%


The choice of the invariant for each leaf that is not yet considered final is not automated and was manually chosen by the authors.
For the initial degenerate tree we selected the Kitano-Suzuki $G$-invariant with $G = \operatorname{SL}_2(\Z_5)$, which resulted
in $54$ distinct values and hence $54$ new leaves, containing equivalence
classes of embeddings with the same value of this invariant.
Of these, $27$ were considered final.
For $24$ of these we were able to show mutual equivalence of all handlebodies represented by the diagrams in each by using
the Reidemeister moves I through V (Figures \ref{fig:reidemeister_a} 
and \ref{fig:reidemeister_b}) and the IH-move
(Fig. \ref{fig:IH-move}), thus obtaining various entries in Tables \ref{tab:table} and \ref{tab:tablered}.

The value $528$ of the mentioned Kitano-Suzuki invariant 
leads to the hard-pair $(6_{12},\seventhirtynine)$, whereas
the values $684$ and $724$ lead to the two hard-pairs $(\sevenfortythree, \sevenfortyfour)$ and $(\sevenforty, \sevenfortyone)$.

Four leaves of the remaining $27$ only contain nonminimal diagrams.  They are manually shown to be nonminimal, since they did not
contain any of the portions listed in Figures \ref{fig:autononminimal}-\ref{fig:autopairs}.

The remaining $23$ leaves of the growing tree $\mathcal T$ after the first invariant computation require stronger invariants in order to
distinguish inequivalent diagrams.
We made here various choices of the invariants: Kitano-Suzuki with $G = S_6$, $\operatorname{SL}_2(\Z_7)$, {\it $G$-image} invariant with $G = A_5$ and
$G = S_6$, the specific choices being indicated in \cite{bigtable:web}.

To further illustrate the dynamic procedure we list here the sequence of invariants leading to entry $7_{10}$ in
Table \ref{tab:table}.
The first invariant computed is $ks_G$, with $G = \operatorname{SL}_2(\Z_5)$ that results in a value of $408$ for a large
number of diagrams.
We further compute on these diagrams the (presumably stronger) invariant $ks_G$ with $G = S_6$ obtaining six distinct values,
hence partitioning the already selected diagrams into six distinct invariant-equivalence%
\footnote{In this case {\it invariant-equivalent} means diagrams that share the value of both $ks_G$ invariants
with $G = \operatorname{SL}_2(\Z_5)$ and $G = S_6$.
Note that there is no guarantee that all diagrams in an equivalence class are equivalent as handlebody-knots.}
classes.

Four of these classes happen to contain mutually equivalent diagrams (shown equivalent by using Reidemeister- and IH-moves).
We found entries $5_3$ and $6_8$ in Table 1 of \cite{IshKisMorSuz:12} and entries $7_4$, $7_{14}$ in our Table \ref{tab:table}.
One class contains entries that could be grouped into two sets of mutually equivalent diagrams that were shown to be inequivalent
in \cite{BePaPaWa:25} and named $7_{59}$, $7_{60}$ in Table \ref{tab:table}.
Finally the class identified by the invariant value of $1272$ required the computation of an even stronger invariant.
We chose to compute the {\it $G$-image} with $G = A_5$, able to separate the diagrams into two classes, each with mutually
equivalent diagrams.
One class turns out to correspond to entry $6_2$ in \cite{IshKisMorSuz:12}, and the other corresponds to entry $7_{10}$
in Table \ref{tab:table}.
We refer to \cite{bigtable:web} for the complete picture.
\draftMP{Mail from Ishii: Our table also contains 69 handlebody-knots. Interestingly, the pairs
$(7_{59},7_{60})$, $(7_{43},7_{44})$, $(7_{40},7_{41})$ are easily distinguishable using
invariants derived from G-family of quandles, so they did not appear as hard
pairs in our study.}
%


\section{Irreducibility}\label{sec:irreducibility}
To divide the handlebody-knots into Tables \ref{tab:table} and \ref{tab:tablered}, we need to   
determine their irreducibility. 
Recall that the \emph{rank} of a handlebody-knot $V$ is the minimal number of generators needed to generate the fundamental group $\pi_1(\Compl V)$. 
We use two ways to compute the presentation of $\pi_1(\Compl V)$: one uses a CW-complex structure of $\Compl V$, while the other is the Wirtinger presentation, and then simplify the presentation as much as possible with the software developed by the third author \cite{appcontour}. As a result, we find all except $7_{68},7_{69}$ in Table  \ref{tab:table} are of rank $3$. The ranks of $7_{68},7_{69}$ are less than or equal to $4$. This allows us to apply the irreducibility test in \cite[Corollary $1.3$]{BePaWa:20b}. 

The Kitano-Suzuki invariant $ks_G$ used in the following Theorem is defined in Section \ref{sec:bigtree}. 

\begin{theorem}\label{thm:irreducibility} 
Given a handlebody-knot $V$ with a rank of $r$ and its $ks_{A_5}$-invariant equal to $k_5$ and 
$ks_{A_4}$-invariant equal to $k_4$, if   
\begin{enumerate}[label=(\roman*)]
\item\label{itm:three_gen} $r\leq 3$, and $k_5$ is not congruent to $17$ modulo $60$, or  
\item\label{itm:four_gen} $r\leq 4$, and $k_4$ is not congruent to $10$ modulo $12$,
\end{enumerate}
then $V$ is irreducible.
\end{theorem} 
\begin{proof}	
By \cite[Corollary $1.3$]{BePaWa:20b},
if $r\leq 3$ and $k_5+14\cdot 4+ 19\cdot 3+ 22 \cdot 5=k_5+223$ is not congruent to $0$ modulo $60$, then $V$ is irreducible. This implies the first assertion. Again by \cite[Corollary $1.3$]{BePaWa:20b}, if $r\leq 4$ and 
$k_4+(6+16k)\cdot 3+(2+6k)\cdot 4$ is not congruent to $0$ modulo $12+24k$, $k=0,1$, then $V$ is irreducible. In other words, if $k_4+26$ and $k_4+98$ are, respectively, not congruent to $0$ modulo $12$ and $36$, then $V$ is irreducible. Note that the former implies the latter and hence the second assertion.
\end{proof}

\begin{corollary}
All handlebody-knots in Table \ref{tab:table} are irreducible. 
\end{corollary}
\begin{proof}
The irreducibility of $7_{68},7_{69}$ is detected by Theorem \ref{thm:irreducibility}\ref{itm:four_gen}, and for all the other handlebody-knots except for $7_{49}$, their irreducibility is detected by Theorem \ref{thm:irreducibility}\ref{itm:three_gen}.   
For the handlebody-knot $7_{49}$, its $ks_{A_5}$-invariant is $77$, and hence the test in Theorem \ref{thm:irreducibility} fails. We observe, however, that $7_{49}$ can be obtained by performing a tunnel looping on a tunnel of the $(5,2)$-torus knot; see Fig.\ \ref{fig:sevenfourtynine_as_looping}. Thus, its irreducibility follows from \cite[Lemma $6.3$; Lemma $5.1$]{Wan:24}. 
\end{proof} 

\begin{figure}
\begin{subfigure}{.5\linewidth}
\centering
\begin{overpic}[scale=.15,percent]{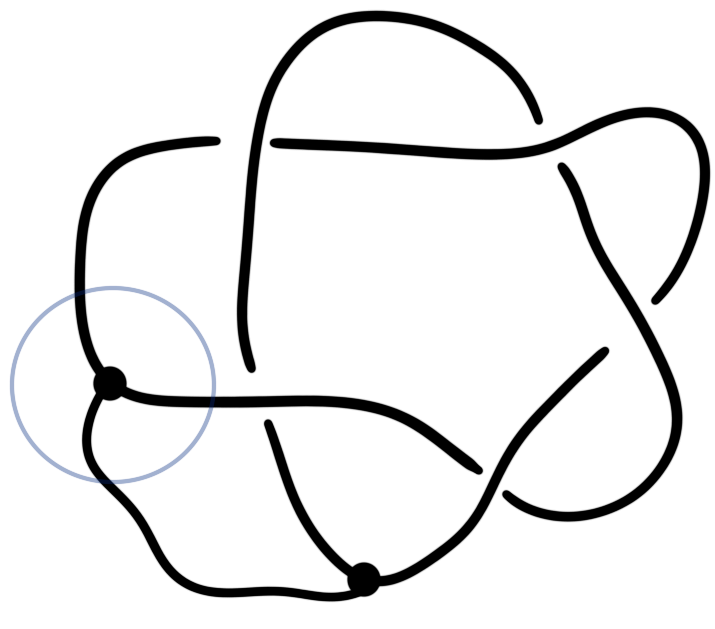} 
	\put(30,75){$K$}
	\put(15,8){$\tau$}
\end{overpic}
\caption{$(5,2)$-torus knot $K$ with tunnel $\tau$.}
\label{}
\end{subfigure}
\begin{subfigure}{.47\linewidth}
\centering
\begin{overpic}[scale=.15,percent]{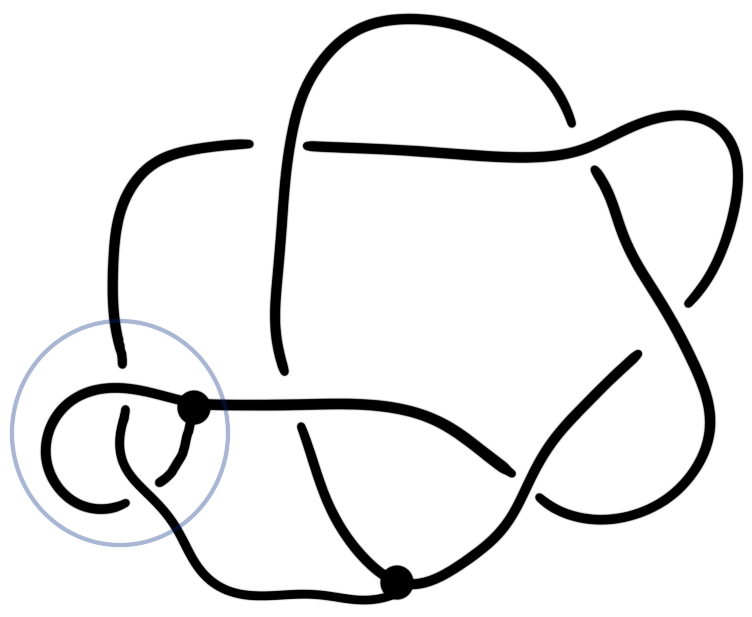}
\end{overpic}
	\caption{$7_{49}$.}
	\label{fig:looping}
\end{subfigure}
\caption{$7_{49}$ as obtained from tunnel looping.}
\label{fig:sevenfourtynine_as_looping} 
\end{figure}

The value of the $ks_{A_5}$- and $ks_{A_4}$-invariants for all handlebodies in Table \ref{tab:table} is included in
the supplementary material \cite{supplement:26}.


\section{Hard pairs}\label{sec:hardpairs}
Here we prove the inequivalence of the two handlebody-knots in each of the pairs:
$(5_1,\sevenfiftytwo)$, 
$(\sevenforty,\sevenfortyone)$,
$(\sevenfortythree, \sevenfortyfour)$. 
Recall from \cite{IshKisOza:15}, a $3$-decomposing sphere of a handlebody-knot $V$ is a $2$-sphere $S$ meeting $V$ at three disks such that $S\cap \Compl V$ is incompressible in $\Compl V$. 
If in addition, the three disks in $S\cap V$ are parallel in $V$,
then $S$ is called a $P_3$-sphere of $V$; see \cite{BelPaoWan:24}.  
A maximal $P_3$-system $\mathcal{S}$ is a set $\{S_1,\dots, S_n\}$ of mutually disjoint $P_3$-spheres such that the intersections $S_i\cap \Compl V$, $i=1,\dots, n$, 
are mutually non-parallel in $\Compl V$. By \cite[Theorem $1.1$]{BelPaoWan:24}, 
if $V$ admits two maximal $P_3$-systems $\mathcal{S},\mathcal{S}'$, then there exists a self-homeomorphism of $\pair$ sending 
$\mathcal{S}$ to $\mathcal{S}'$.

\begin{figure}[t]
	\begin{subfigure}{.48\linewidth}
		\centering
		\begin{overpic}[scale=.15,percent]{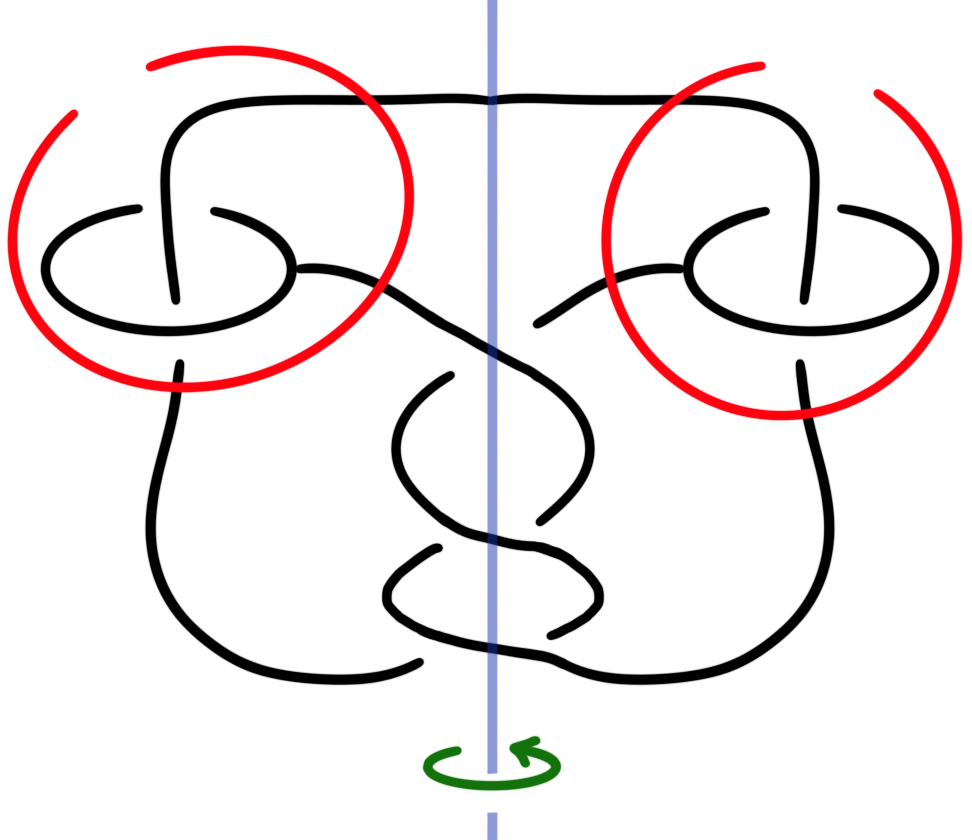}
			\put(7.5,76){$S_1$}
			\put(81,77){$S_2$}
			\put(60,5){$\pi$}
			\put(30,67){$X$}
			\put(53,64){$Y$}
			\put(70,68){$Z$}
			\put(45,78){$a$}
			\put(39,49){$b$}
			\put(18,30){$c$}
		\end{overpic}
		\caption{$P_3$-surfaces of $\sevenfortythree$.}
		\label{fig:three_decomp:fortythree}
	\end{subfigure} 
	\begin{subfigure}{.48\linewidth}
		\centering
		\begin{overpic}[scale=.15,percent]{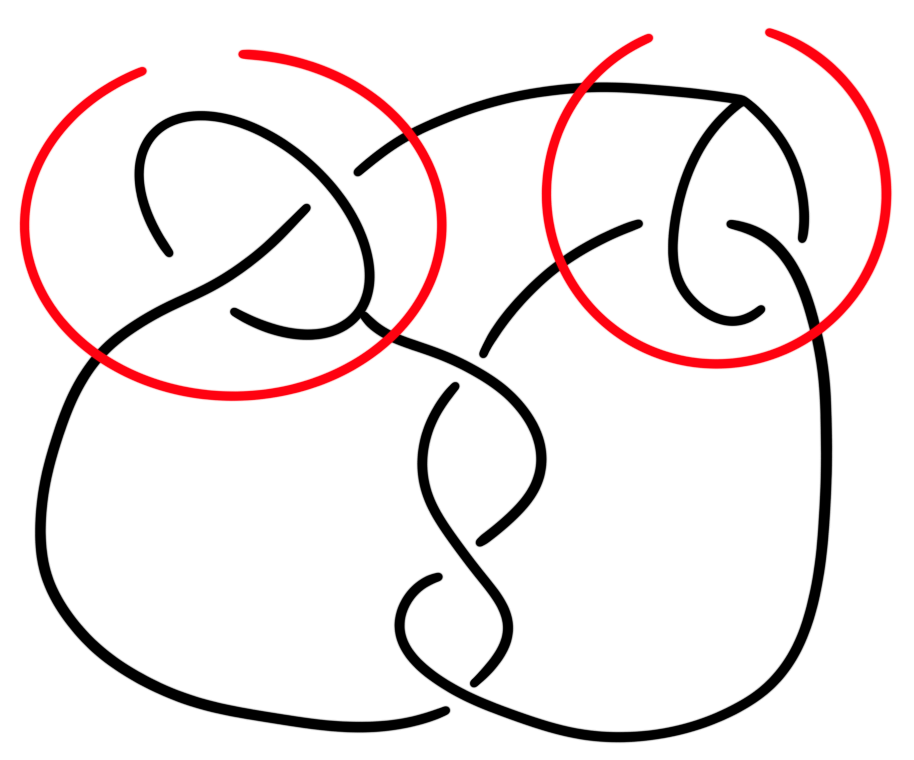}
			\put(18,75){$S_1'$}
			\put(73,77){$S_2'$}
			\put(21,60){$X'$}
			\put(51,58){$Y'$}
			\put(68,65){$Z'$}
			\put(50,74){$a'$}
			\put(49,47){$b'$}
			\put(8,28){$c'$}
		\end{overpic}
		\caption{$P_3$-surfaces of $\sevenfortyfour$.}
		\label{fig:three_decomp:fortyfour}
	\end{subfigure} 
	\nada{
		\begin{subfigure}{.48\linewidth}
			\centering
			\begin{overpic}[scale=.12,percent]{sg_twoone}
			\end{overpic}
			\caption{Spatial graphs $\Gamma_X,\Gamma_Z,\Gamma_{X'},\Gamma_{Z'}$.}
			\label{fig:sg_gamma_XZ}
		\end{subfigure} 
		\begin{subfigure}{.48\linewidth}
			\centering
			\begin{overpic}[scale=.12,percent]{sg_threeone}
			\end{overpic}
			\caption{Spatial graphs $\Gamma_Y,\Gamma_{Y'}$.}
			\label{fig:sg_gamma_Y}
		\end{subfigure}
	}
	\caption{}
\end{figure}

\begin{theorem} 
	$\sevenfortythree$, $\sevenfortyfour$ are inequivalent.
\end{theorem}
\begin{proof}
	Denote by $V$ (resp.\ by $V'$) the handlebody-knot $\sevenfortythree$ (resp.\  $\sevenfortyfour$), and note that the spheres $S_1,S_2$ (resp.\ $S_1',S_2'$) in Fig.\ \ref{fig:three_decomp:fortythree} (resp.\ Fig.\ \ref{fig:three_decomp:fortyfour}) both are $P_3$-spheres. 
	Observe that the union $S_1\cup S_2$ (resp.\ $S_1'\cup S_2'$) cuts $\sphere$ into three components $X,Y,Z$ (resp.\ $X',Y',Z'$) as shown in Fig.\ \ref{fig:three_decomp:fortythree}, (resp.\ Fig.\ \ref{fig:three_decomp:fortyfour}). 
	
	Consider the unions $\Compl X\cup V$, $\Compl Y\cup V$, and $\Compl Z\cup V$ (resp.\ $\Compl {X'}\cup V',\Compl {Y'}\cup V'$, and $\Compl {Z'}\cup V'$), and observe that they are all genus two handlebodies. Additionally, the disks in the intersection $V\cap (S_1\cup S_2)$ (resp.\ $V'\cap (S_1'\cup S_2')$) induce a spine for each of them, and hence induce three spatial graphs $\Gamma_X,\Gamma_Y,\Gamma_Z$ (resp.\ $\Gamma_{X'},\Gamma_{Y'},\Gamma_{Z'}$). 
	
	The spatial graphs $\Gamma_X, \Gamma_Z, \Gamma_{X'}$, $\Gamma_{Z'}$
	are the spatial handcuff graph $\op{G2}$ in Table \ref{tab:graphs}, namely $2_1$ in \cite{Mor:07b}, while 
	$\Gamma_Y,\Gamma_{Y'}$ are the mirror image of  
	$\op{G3}$ in Table \ref{tab:graphs}, namely 
	$3_1$ in \cite{Mor:09a}, so they are all prime. This implies that $\{S_1,S_2\}$ (resp.\  $\{S_1',S_2'\}$) is a maximal $P_3$-system.
	
	Label the arcs in $\Gamma_Y$ and $\Gamma_{Y'}$ by $a,b,c$ and $a',b',c'$, respectively; see Figs.\ \ref{fig:three_decomp:fortythree}, \ref{fig:three_decomp:fortyfour}.
	Then, by the uniqueness of the maximal $P_3$-system \cite[Theorem $1.1$]{BelPaoWan:24} 
	and the fact that 
	there is a self-homeomorphism of $\pair$ swapping $S_1,S_2$ (see Fig.\ \ref{fig:three_decomp:fortythree}), there is a self-homeomorphism $f$ of $\pair$ with $f(S_i)=S_i'$, $i=1,2$. 
	In particular, $f$ induces an equivalence $g$ from 
	$\Gamma_Y$ to $\Gamma_{Y'}$ with $g(a)=c'$, $g(b)=b'$, and $g(c)=a'$. This contradicts the fact that the constituent knot $b\cup c$ of $\Gamma_Y$ is a trefoil knot, whereas the constituent knot $b'\cup a'$ of $\Gamma_{Y'}$ is trivial; 
	see $\op{G3}$ in Table \ref{tab:graphs}.   
\end{proof}

\begin{figure}[t]
	\begin{subfigure}{.33\linewidth}
		\centering
		\begin{overpic}[scale=.12,percent]{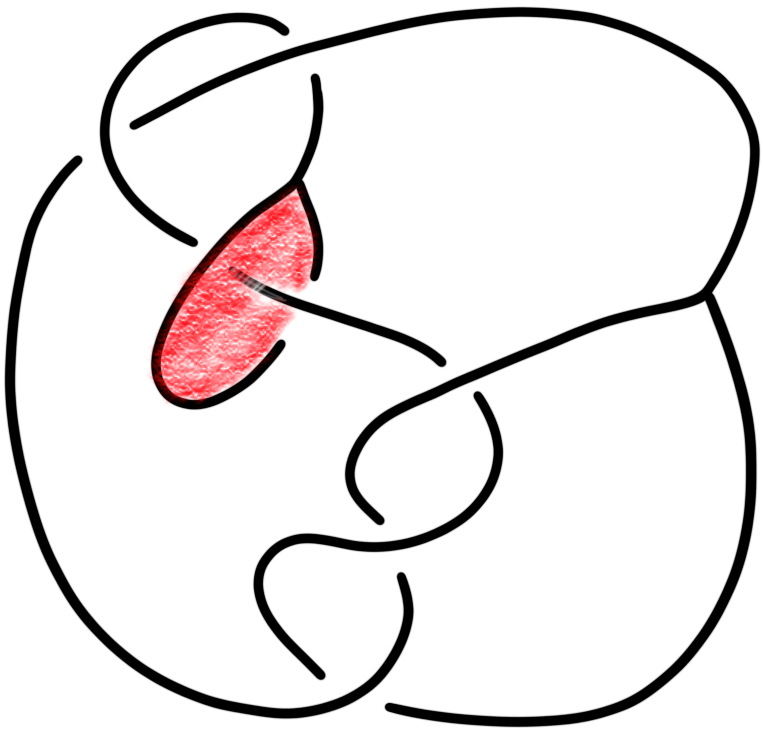}
			\put(24,50){$A$}
		\end{overpic}
		\caption{$\sevenforty$ and annulus $A$.}
		\label{fig:annulus_forty}
	\end{subfigure} 
	\begin{subfigure}{.33\linewidth}
		\centering
		\begin{overpic}[scale=.12,percent]{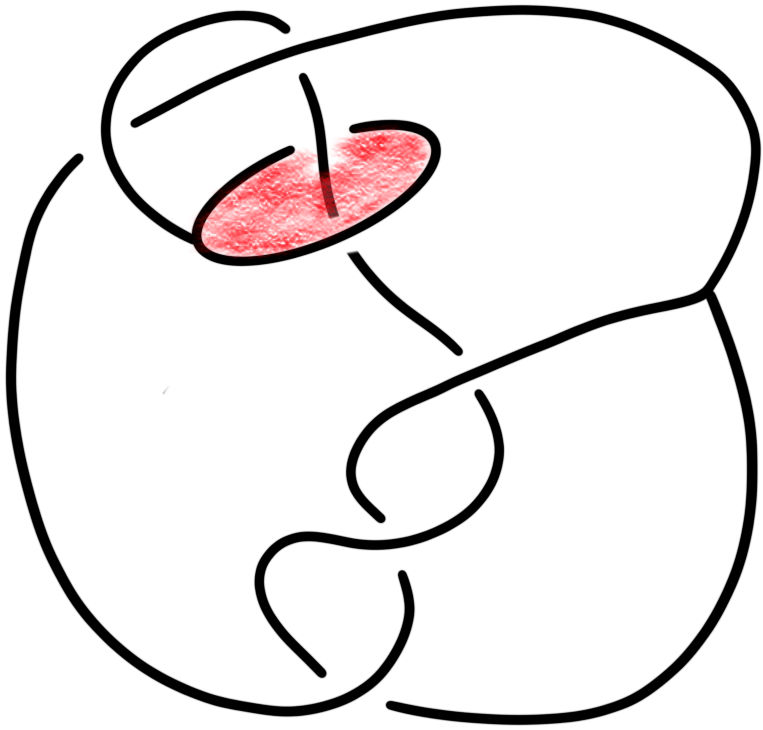}
			\put(28,64){$A'$}
		\end{overpic}
		\caption{$\sevenfortyone$ and annulus $A'$.}
		\label{fig:annulus_fortyone}
	\end{subfigure} 
	\begin{subfigure}{.32\linewidth}
		\centering
		\begin{overpic}[scale=.15,percent]{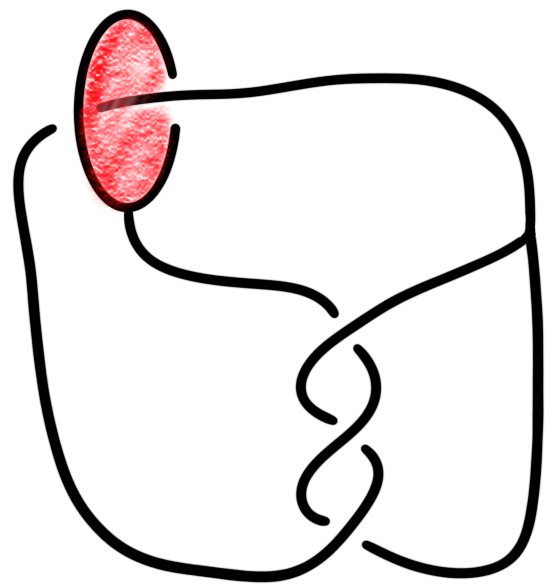}
			\put(18,87){\tiny $A_0$}
		\end{overpic}
		\caption{$\fiveone$ and annulus $A_0$.}
		\label{fig:fiveone}
	\end{subfigure}
	\caption{}
\end{figure}
\begin{theorem}
	$\sevenforty$ and $\sevenfortyone$ are inequivalent.
\end{theorem} 
\begin{proof}
	Denote by $V,V'$ the handlebody-knots $\sevenforty,\sevenfortyone$, respectively, and observe that 
	$\Compl V$ (resp.\ $\Compl{V'}$) admits an essential annulus $A$ (resp.\ $A'$) as shown in Fig.\ \ref{fig:annulus_forty} (resp.\   \ref{fig:annulus_fortyone}). 
	Let $\rnbhd{A}$ (resp.\ $\rnbhd{A'}$) be a regular neighborhood of $A$ in $\Compl V$ (resp.\ $A'$ in $\Compl{V'}$).
	Then the union $V\cup \rnbhd{A}$ (resp.\ $V'\cup \rnbhd{A'}$) in $\sphere$ is the handlebody-knot $5_1$, denoted by $V_0$; see Fig.\ \ref{fig:fiveone}. By \cite[Theorem $1.4$]{Wan:24}, 
	the exterior $\Compl{V_0}$ admits a unique essential annulus $A_0$.
	
	\subsection*{Claim: $A$ (resp.\ $A'$) is the unique annulus in $\Compl V$ (resp.\ $\Compl {V'}$).}
	Suppose otherwise and $\Compl V$ admits an essential annulus $A^\ast$ non-isotopic to $A$. By the JSJ-decomposition of $\Compl V$, if $A$ is not a characteristic annulus, we can take $A^\ast$ to be a characteristic annulus. Thus, whether $A$ is characteristic or not, $A^\ast$ can always be isotoped away from $A$ and hence from a regular neighborhood $\rnbhd{A}$ of $A$. In particular, $A^*$ can be regarded as an annulus in the exterior $\Compl{V_0}$, and 
	the incompressibility of $A^*\subset\Compl V$ implies that it is incompressible also in $\Compl{V_0}$.
	
	Let 
	$A_+,A_-$ be the components in the frontier of $\rnbhd{A}\subset\Compl V$.  
	Suppose $A^*$ is $\partial$-compressible in $\Compl{V_0}$, and $D$ is a 
	$\partial$-compressing disk. Then, since $\fiveone$ is irreducible, $\Compl{V_0}$ is $\partial$-irreducible. Therefore compressing $A^\ast$ along $D$ yields an inessential disk in $\Compl{V_0}$. This, along with the incompressibility of $A^*$, shows that $A^*$ is $\partial$-parallel 
	to an annulus $A^{**}$ in $\partial V_0$. 
	Since $A^*$ is essential in $\Compl V$,
	$A^{**}$ meets $A_+\cup A_-$. Given $\partial A^\ast$ does not meet $A_+\cup A_-$, $A^{**}$ either contains both $A_+, A_-$ or only one of them. If $A^{**}$ contains $A_+\cup A_-$, the cores of $A_+,A_-$ are parallel in $\partial V_0$, a contradiction. If $A^{**}$ contains $A_+$ or $A_-$, then $A^{**}$ and $A_+$ or $A_-$ has the same core, and hence $A^{*}$ is isotopic to $A_+$ or $A_-$ in $\Compl{V_0}$, and isotopic to $A$ in $\Compl{V}$, a contradiction. 
	
	As a result, $A^*\subset \Compl {V_0}$ is $\partial$-incompressible and hence essential, yet this contradicts that $A_0$ is the unique essential annulus in $\Compl{V_0}$ and $\partial A_0$ meets $A_+\cup A_-$.
	The same argument applies to $V'$.  
	
	Suppose $\sevenforty,\sevenfortyone$ are equivalent, up to mirror image, and $f:\pair\rightarrow \pairprime$ is a homeomorphism. Then by the uniqueness of $A,A'$, it may be assumed that $f(A)=A'$ and hence $f$ induces a self-homeomorphism 
	of $\fiveone$ that reverses the orientation of $\partial A_0$, contradicting that every self-homeomorphism of $\fiveone$ is isotopic to the identity by \cite[Theorem $1.2$]{Wan:22}.   
\end{proof}

\begin{figure}[t]
	\begin{subfigure}{.48\linewidth}
		\centering
		\begin{overpic}[scale=.14,percent]{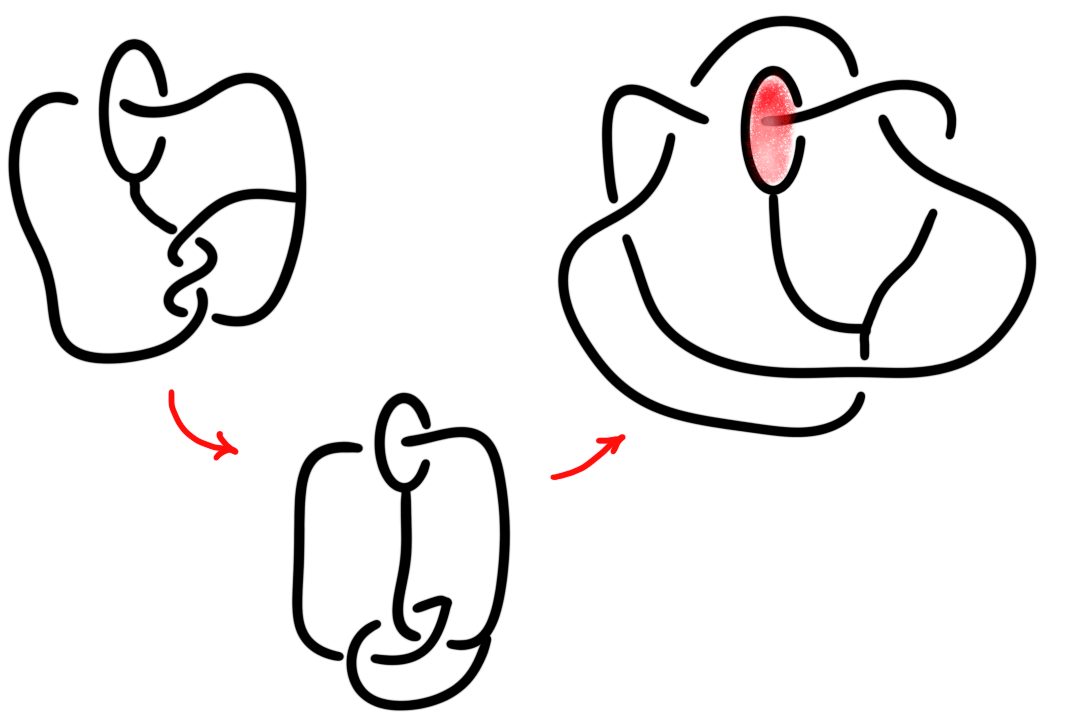}
			\put(70,52){\footnotesize $A$}
		\end{overpic}
		\caption{Deforming $5_1$; annulus $A$.}
		\label{fig:deformation_fiveone}
	\end{subfigure} 
	\begin{subfigure}{.48\linewidth}
		\centering
		\begin{overpic}[scale=.14,percent]{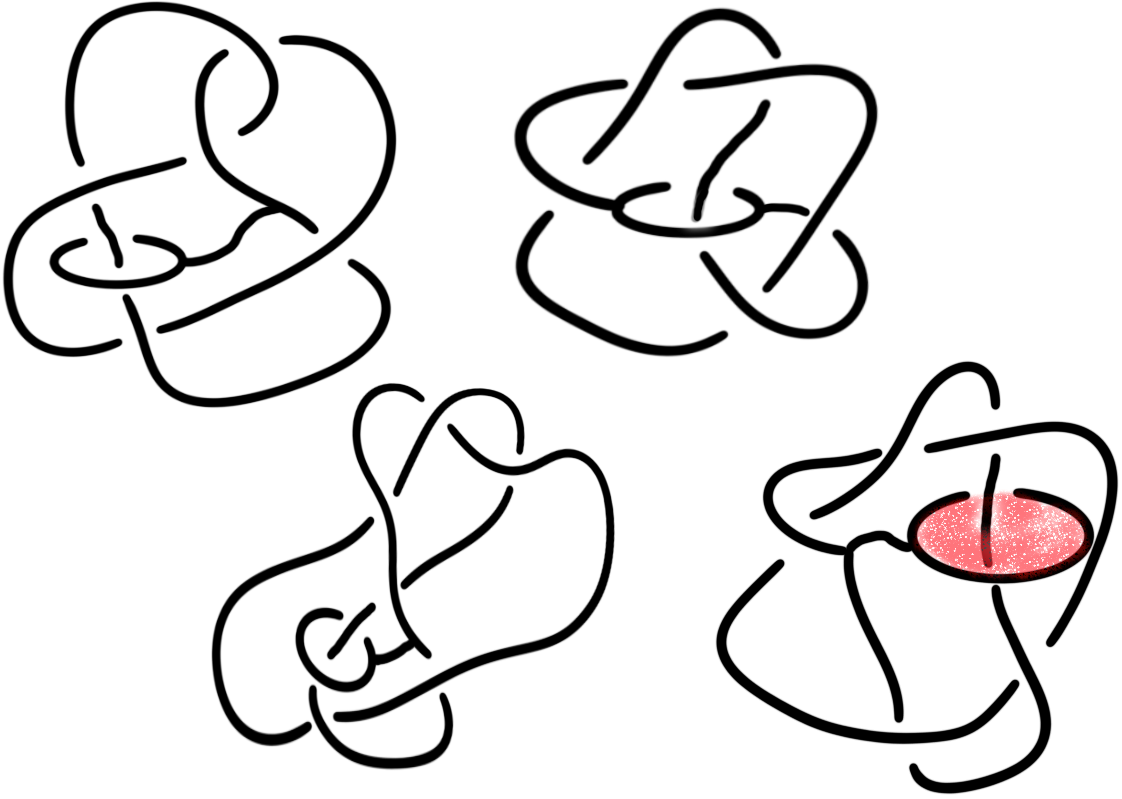}
			\put(90,21.5){\footnotesize $A'$}
		\end{overpic}
		\caption{Deforming $\sevenfiftytwo$; annulus $A'$.}
		\label{fig:deformation_sevenfifytwo}
	\end{subfigure} 
	\begin{subfigure}{.48\linewidth}
		\centering
		\begin{overpic}[scale=.13,percent]{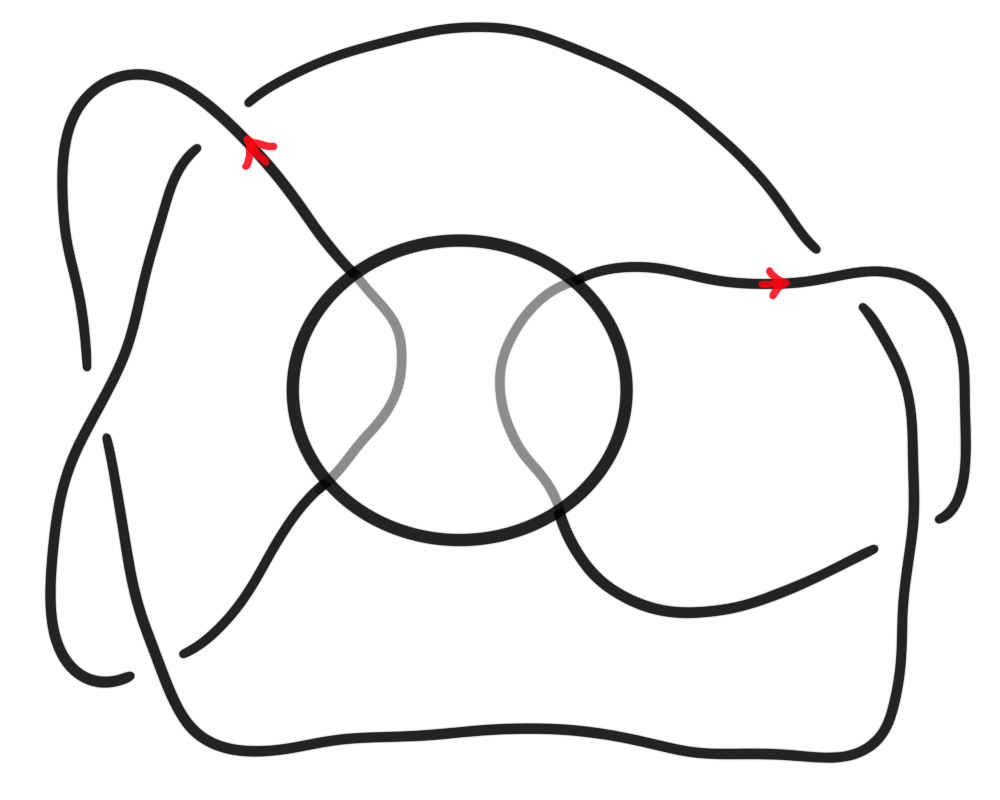}
			\put(70,53){$a$}
			\put(25,56){$b$}
			\put(42,29){$B_0$}
			\put(55,55){\tiny NE}
			\put(33.5,56){\tiny NW}
			\put(58.5,26){\tiny SE}
			\put(22,30){\tiny SW}
		\end{overpic}
		\caption{$5_1$: induced tangle.}
		\label{fig:tangle_fiveone}
	\end{subfigure} 
	\begin{subfigure}{.48\linewidth}
		\centering
		\begin{overpic}[scale=.13,percent]{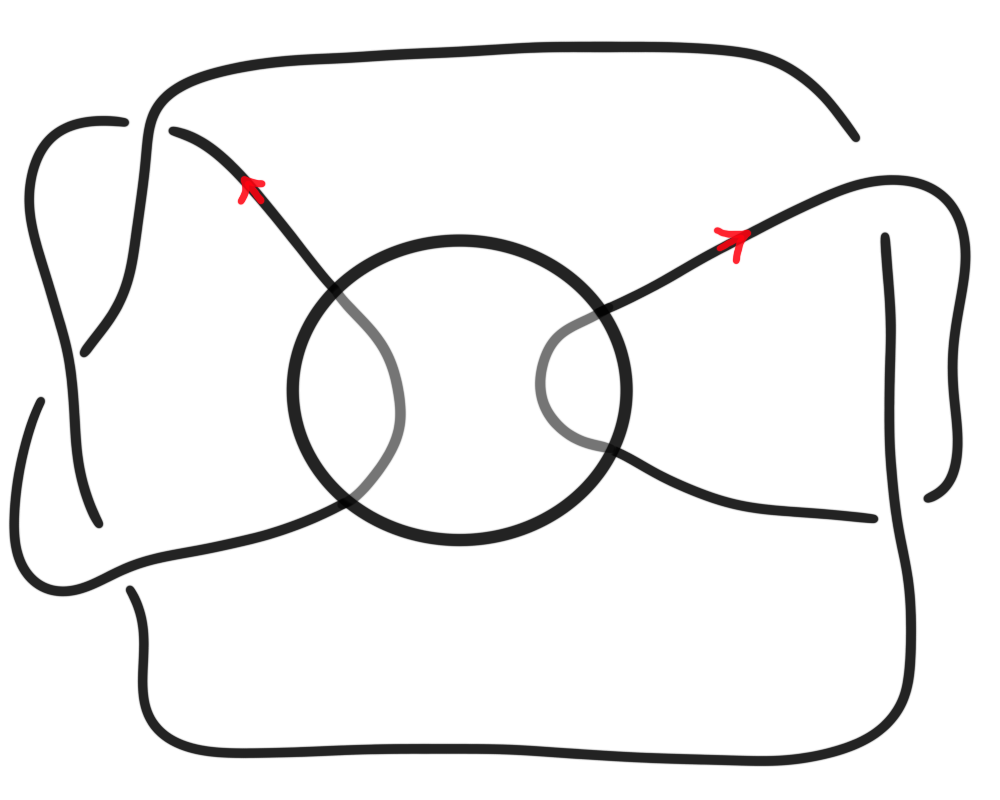}
			\put(75.5,51.5){$a'$}
			\put(17.5,56.5){$b'$}
			\put(42,29){$B_0'$}
			\put(56.5,53){\tiny NE'}
			\put(31,56){\tiny NW'}
			\put(59,27.5){\tiny SE'}
			\put(30,22.5){\tiny SW'}
		\end{overpic}
		\caption{$\sevenfiftytwo$: induced tangle.}
		\label{fig:tangle_sevenfiftytwo}
	\end{subfigure} 
	\caption{}
\end{figure}

\begin{theorem}
	$5_1$ and $\sevenfiftytwo$ are inequivalent.
\end{theorem} 
\begin{proof}
	Denote by $V,V'$ the handlebody-knots $5_1, \sevenfiftytwo$, respectively, and observe that their exteriors $\Compl V, \Compl {V'}$ admit essential
	annulus $A, A'$ in Figs.\ \ref{fig:deformation_fiveone}, 
	\ref{fig:deformation_sevenfifytwo}), respectively. Since $\partial A, \partial A'$ both have exactly one component bounding a non-separating disk in $V,V'$, respectively, $A\subset\Compl V,  A'\subset \Compl{V'}$ are the unique essential annulus by \cite[Theorem $1.4$]{Wan:24}, up to isotopy. Denote by 
	$D\subset V$ (resp.\ $D'\subset V'$) the disk bounded by the component of $\partial A$ (resp.\ $\partial A'$), respectively.   
	
	Now, Suppose $5_1$ and $\sevenfiftytwo$ are equivalent, up to mirror image, and $f:\pair\rightarrow \pairprime$ 
	is a homeomorphism. Then it may be assumed that $f(A)=A'$, and hence $f$ induces a self-homeomorphism $g$ of $\sphere$ sending $V_A:=V\cup \rnbhd{A}$ to $V_A':=V'\cup\rnbhd{A'}$.
	Let $A_a,A_b$ (resp.\ $A_a',A_b'$) be the two annuli in the frontier of $\rnbhd{A}\subset \Compl V$ (resp.\ $\rnbhd{A'}\subset \Compl{V'}$).
	Then the cores of $A_a,A_b$ (resp.\ $A_a',A_b'$) bound two disks $D_a,D_b$ in $V_A$ (resp.\ $D_a',D_b'$ in $V_A'$).
	
	Note that cutting $V_A$ (resp.\ $V_A'$) along one of $D_a,D_b$, say $D_a$ (resp.\ one of $D_a',D_b'$, say $D_a'$), gives a regular neighborhood of the right-hand trefoil, while cutting $V_A$ along $D_b$ (resp.\ $V_A'$ along $D_b'$) gives a regular neighborhood of a trivial knot.
	This implies $g$ is orientation-reversing.
	
	Denote by $a,b\subset V_A$ and by $a',b'\subset V_A'$ the arcs dual to $D_a,D_b$ and $D_a',D_b'$, respectively. Orient the arcs $a,b$ (resp.\ $a',b'$) so they point away from $D$ (resp.\ $D'$); see Fig.\ \ref{fig:tangle_fiveone} (resp.\ \ref{fig:tangle_sevenfiftytwo}). 
	In particular, the union $a\cup b$ (resp.\ $a'\cup b'$) can be regarded as a $2$-string tangle in the exterior $B$ of a regular neighborhood $B_0$ of $A\cup D$ 
	(resp.\ the exterior $B'$ of a regular neighborhood $B_0'$ of $A'\cup D'$). 
	It may be assumed that $g(B_0)=B_0'$. Also, by our choice of orientation and $D_a,D_b,D_a',D_b'$, we have $g(a)=a',g(b)=b'$ with the orientation of $a,a',b,b'$ respected.
	
	Let NW,NE,SW,SE (resp.\ NW',NE',SW',SE') be the four boundary points of $a\cup b$ (resp.\ $a'\cup b'$) as shown in Fig.\ \ref{fig:tangle_fiveone} (resp.\ \ref{fig:tangle_sevenfiftytwo}), and observe that there is a $2$-sphere $S.S'$ 
	meeting $B,B'$ at a disk $E,E'$ and $b,b'$ at two points, respectively.
	Let $m:=\partial E, m':=\partial E'$ be the meridian of $B,B'$ and fix some longitude $l,l'$ for $B,B'$, respectively. 

	Since $g(a)=a',g(b)=b'$ with their orientation preserved, $g$ restricts to a homeomorphism from $B_0$ to $B_0'$ carrying XX to XX', where
	XX$=$NW,NE,SW or SE. Consider the rational tangle $t\subset B_0$ with a slope $\infty$ with respect to $m,l$; see Fig.\ \ref{fig:tangle_fiveone}. 
	Then $g$ being a homeomorphism implies $g(t)=t'\subset B'$ is rational. Suppose $t'$ has a slope of $\frac{p}{q}$ with respect to $m',l'$. Since $L=t\cup a\cup b$ is a link, $L':=g(L)=t'\cup a'\cup b'$ is also a link, so the two points NE' and SE' are joined by an arc in $t'$. This implies $q$ is even.  
	
	If $q\neq 0$, then $L'$ is a Montesinos link with three tangles $T_1,T_2,T_3$ with a slope of $-\frac{1}{3}$, $\frac{p}{q}$, $\frac{1}{2}$, respectively; see Fig.\ \ref{fig:tangle_sevenfiftytwo}. In particular, the branched cover along $L'$
	is Seifer fibered over a $2$-sphere with $3$ exceptional fibers, so $L'$ is prime, contradicting that $g(L)=L'$ and $L$ is a composite link.  
	As a result, we have $q=0$, and hence $(B_0',t_0')$ has a slope of $\infty$ with respect to $m',l'$.  
	
	Orient $L,L'$ using the orientation of $a\cup b,a'\cup b'$, respectively. Then
	$g$ preserves the orientation of $L,L'$, yet the linking number of both $L,L'$ is $1$, contradicting $g$ reverses the orientation of $\sphere$.  
\end{proof}

\section*{Acknowledgements}
The first author
is member of the GNAMPA-INdAM.
The fourth author gratefully acknowledges the support from NSTC, Taiwan (grant no. 112-2115-M-110 -001 -MY3).



\begin{thebibliography}{99}
%

\nada{
\bibitem{BeBePaPa:15}
G. Bellettini, V. Beorchia, M. Paolini, F. Pasquarelli,
\textit{Shape Reconstruction from Apparent Contours.
Theory and Algorithms,}
Computational Imaging and Vision,
Springer-Verlag (2015) pp. iii-333.
}

\bibitem{BePaPaWa:23}
G. Bellettini, G. Paolini, M. Paolini, Y.-S. Wang,
\textit{A table of $n$-component handlebody links of genus $n+1$ up to six crossings},
Math. Proc. Camb. Phil. Soc. {\bf 174} (2023), 199--223.

\bibitem{BePaPaWa:25}
G. Bellettini, G. Paolini, M. Paolini, Y.-S. Wang,
\textit{Handlebody-knots obtained by tangle replacement},
\href{https://arxiv.org/abs/2511.07796v3}{arXiv:2511.07796v3}.

\bibitem{supplement:26}
G. Bellettini, G. Paolini, M. Paolini, Y.-S. Wang,
\textit{Supplementary computational material for ``A table of genus two handlebody-knots with seven crossings''},
Zenodo, \url{https://doi.org/10.5281/zenodo.21512666} (2026).


\bibitem{BePaWa:20a}
G. Bellettini, M. Paolini, Y.-S. Wang,
\textit{A complete invariant for connected surfaces in the $3$-sphere},
J. Knot Theory Ramifications {\bf 29} (2020), 1--24.


\nada{
\bibitem{BePaWa:21}
G. Bellettini, M. Paolini, Y.-S. Wang,
\textit{On closed oriented surfaces in the 3-sphere},
J. Knot Theory Ramifications {\bf 21}, 1--25.
}

\bibitem{BePaWa:20b}
G. Bellettini, M. Paolini, Y.-S. Wang,
\textit{Numerical irreducibility criteria for handlebody links}, 
Topology Appl.
{\bf 284} (2020), 1--14. 

\bibitem{BelPaoWan:24}
G. Bellettini, M. Paolini, Y.-S. Wang,
\textit{Unique $3$-decomposition and mapping classes of knotted handlebodies},
Geom. Dedicata 
{\bf 219}(74) (2025). 


\bibitem{Buo:03} S. Buoncristiano,
\textit{Fragments of Geometric Topology from the Sixties},
Geometry and topology monographs {\bf 6}, University of Warwick, Mathematics Institute (2003).

 
\bibitem{ChoMcC:09} S. Cho, D. McCullough:
\textit{The tree of knot tunnels},
Geom. Topol. {\bf 13} (2009) 769--815 769 
\nada{
\bibitem{Cro:04} P. R. Cromwell,
\textit{Knots and Links},
Cambridge University Press, Cambridge, (2004). 
 
\bibitem{DavShe:01} R.J. Daverman, R.B. Sher,
\textit{Handbook of Geometric Topology}
North-Holland, Amsterdam (2001).

\bibitem{Dehn:11}
M. Dehn,
\textit{\"Uber unendliche diskontinuierliche Gruppen},
Math. Ann. {\bf 71} (1911), 116--144.


\bibitem{Eps:66} D. B. A. Epstein,
\textit{Curves on 2-manifolds and isotopies},
Acta Math. {\bf 115} (1966), 83--107.
 




\bibitem{FriMilPow:17} S. Friedl, A. N. Miller, M. Powell,
\textit{Determinants of amphichiral knots}, 
arXiv:1706.07940 [math.GT]. 


\bibitem{FomMat:97}
A. T. Fomenko, S. V. Matveev,
\textit{Algorithmic and Computer Methods for Three-Manifolds},
Kluwer Academic Publisher, (1997).




\bibitem{Fox:48} R. H. Fox,
\textit{On the imbedding of polyhedra in 3-space}, Ann. of Math. {\bf 2}(49) (1948), 462--470.


\bibitem{Fox:50} R. H. Fox,
\textit{Recent development of knot theory at Princeton}, Proceedings of the International Congress of Mathematicians
{\bf 2} (1950), 453--458.

\bibitem{Fox:52} R. H. Fox,
\textit{On the complementary domains of a certain pair of inequivalent knots}, Indagationes Mathematicae (Proceedings) {\bf 55} (1952), 37--40.
  
\bibitem{FunKod:20} K. Funayoshi, Y. Koda,
\textit{Extending automorphisms of the genus-2 surface over the $3$-sphere}, 
Q. J. Math. {\bf 71} (2020), 175--196.
 
\bibitem{GorLue:89}
C. Gordon, J. Luecke,
\textit{Knots are determined by their complements},
J. Amer. Math. Soc. 
{\bf 2} (1989), 371--415.

\bibitem{Gro:69} J. L. Gross,
\textit{A unique decomposition theorem for $3$-manifold with connected boundary},
Trans. Amer. Math. Soc. {\bf 142} (1969), 191--199. 

\bibitem{Gro:70} J. L. Gross,
\textit{A unique decomposition theorem for $3$-manifold with several boundary components},
Trans. Amer. Math. Soc. {\bf 147} (1970), 561--572. 
 


\bibitem{Gru:40} I. A. Grushko,
\textit{On the bases of a free product of groups},
Matematicheskii Sbornik, {\bf 8} (1940), 169--182.


\bibitem{Hat:83} A. Hatcher,
\textit{A proof of the Smale Conjecture}
Ann. of Math. {\bf 117} (1983), 553--607. 
  



\bibitem{Hei:71} W. Heil,
\textit{On Kneser's conjecture for bounded $3$-manifolds},
Proc. Comb. Phil. Soc. {\bf 71} (1972), 71--30.



\bibitem{Hem:04} J. Hempel,
\textit{3-manifolds}, AMS Chelsea Publishing, Providence, RI, (2004). 

 
\bibitem{HosThiWee:98} J. Hoste, M. Thistlethwaite, J. Weeks,
\textit{The first 1,701,936 knots},
Math. Intelligencer {\bf 20} (1998), 33--48. 
}  
\bibitem{Ish:08}
A. Ishii,
\textit{Moves and invariants for knotted handlebodies},
Algebr. Geom. Topol.
{\bf 8} 
(2008), 1403--1418.  

\nada{ 
\bibitem{IshIwa:12}
A. Ishii, M. Iwakiri, 
\textit{Quandle cocycle invariants for spatial graphs and knotted handlebodies},
Canad. J. Math. {\bf 64} (2012), 102--122. 


\bibitem{IshIwaJanOsh:13}
A. Ishii, M. Iwakiri, Y. Jang, K. Oshiro,
\textit{A G-family of quandles and handlebody-knots},
Illinois J. Math. {\bf 57} (2013), 817--838. 
} 
  
\bibitem{IshKisMorSuz:12}
A. Ishii, K. Kishimoto, H. Moriuchi, M. Suzuki,  
\textit{A table of genus two handlebody-knots up to six crossings},
J. Knot Theory Ramifications {\bf 21}
(2012), 1250035 (9 pages).  

\bibitem{IshKisOza:15}
A. Ishii, K. Kishimoto, M. Ozawa,
\textit{Knotted handle decomposing spheres for handlebody-knots},
J. Math. Soc. Japan {\bf 67}
(2015),
407--417. 



\nada{

\bibitem{Joh:91} K. Johannson,
\textit{Topology and combinatorics of 3-manifolds},
Lecture notes in Mathematics {\bf 1599}, 
Springer-Verlag Berlin Heidelberg, (1991).







\bibitem{Kau:89} L. H. Kauffman, 
\textit{Invariants of graphs in three-space},
Trans. Amer. Math. Soc. {\bf 311} (1989) (2), 679--710.  
}
 
\bibitem{KitSuz:12}
T. Kitano, M. Suzuki,
\textit{On the number of $SL(2,\mathbb{Z}/p\mathbb{Z})$-representations of 
knot groups},
J. Knot Theory Ramifications {\bf 21}
\draftGB{pages missing}
(2012).

\nada{
\bibitem{KodOzaGor:15} Y. Koda, M. Ozawa, with an appendix by C. Gordon,  
\textit{Essential surfaces of non-negative Euler characteristic in genus two handlebody exteriors},
Trans. Amer. Math. Soc. {\bf 367} (2015), no. 4, 2875--2904.

 
\bibitem{Kur:60} A. G. Kurosh,
\textit{The Theory of Groups}  
Translated from the Russian and edited by K. A. Hirsch. 2nd English ed. 2 volumes Chelsea Publishing Co., New York (1960) Vol. 1: 272 pp. Vol. 2: 308 pp.
 
 
\bibitem{Lac:09} M. Lackenby,
\textit{The crossing number of composite knots},
J. Topol. {\bf 2} (4) (2009), 747--768.
}

  
\bibitem{LeeLee:12} J. H. Lee, S. Lee, 
\textit{Inequivalent handlebody-knots with homeomorphic complements},
Algebr. Geom. Topol. {\bf 12} (2012), 1059–-1079. 
 
%
\nada{
\bibitem{Mal:18} A. V  Malyutin,
\textit{On the Question of Genericity of Hyperbolic Knots},
Int. Math. Res. Notices (2018), 1073-7928.





 
\bibitem{Man:00}
V. Manturov,
\textit{Knot Theory},
Chapman \& Hall/CRC, 2000.

 
%
 
 
\bibitem{Maz:61} B. C. Mazur,
\textit{On embeddings of spheres},
Acta Math. {\bf 105} (1961) 1–-17.

 

  
 
 
\bibitem{Mel:06} P. Melvin,
\textit{A Topological Menagerie},
Amer. Math. Monthly, {\bf 113} (2006) 348--351.
 
 

\bibitem{Miz:13} A. Mizusawa,
\textit{Linking numbers for handlebody-links},
Proc. Japan Acad. {\bf 89}, Ser. A (2013), 60--62. 

 
\bibitem{Moi:52} E. E. Moise,
\textit{Affine structures in $3$-manifolds: II. Positional Properties of 2-Spheres},
Ann. of Math. {\bf 55} 
(1952), 172--176

 
\bibitem{Moi:77} E. E. Moise,
\textit{Geometric Topology in Dimensions $2$ and $3$}, Springer-Verlag, Grad. Texts in Math. {\bf 47}, (1977).
}

\bibitem{Mor:09a} H. Moriuchi,
\textit{An enumeration of theta-curves with up to seven crossings},
J. Knot Theory Ramifications {\bf 18} 
(2009), 67--197.  

\bibitem{Mor:07b} H. Moriuchi,
\textit{A table of handcuff graphs with up to seven crossings},
OCAMI Studies Vol 1. Knot Theory for Scientific objects (2007) 179--300.

\bibitem{Mor:09b} H. Moriuchi,
\textit{A table of $\theta$-curves and handcuff graphs with up to 
seven crossings}, 
Adv. Stud. Pure Math. Noncommutativity and Singularities: Proceedings of French--Japanese symposia held at IHES in 2006, J.-P. Bourguignon, M. Kotani, Y. Maeda and N. Tose, eds. (Tokyo: Mathematical Society of Japan, 2009), 281--290. 
 
  
\bibitem{Mot:90} M. Motto,
\textit{Inequivalent genus two handlebodies in $S^{3}$ with homeomorphic complements}, Topology Appl. {\bf 36}(3) (1990), 283--290

\nada{
\bibitem{Mun:59} J. Munkres,
\textit{Obstructions to the smoothing of piecewise-differentiable homeomorphisms},
Bull. Amer. Math. Soc. {\bf 65} (5) (1959), 332--334.

 
\bibitem{Mun:66} J. Munkres,
\textit{Elementary Differential Topology},
(AM-54), Princeton University Press, Princeton, NJ (1966).
  
\bibitem{Neu:43} B. H. Neumann,
\textit{On the number of generators of a free product}, 
Journal of the London Mathematical Society {\bf 18} (1943), 12--20. 
 
\bibitem{Och:91} M. Ochiai,
\textit{Heegaard diagrams of 3-manifolds},
Trans. Amer. Math. Soc. {\bf 328}
(1991), 863--879.
 
}

\bibitem{appcontour} M. Paolini,
Appcontour. Computer software. Vers. 2.5.8.
Apparent contour. (2018) $<$\url{git@github.com:appcontour/appcontour.git/}$>$.

\bibitem{bigtable:web} M. Paolini (2024, July 24).
{\it Complete table of handlebody knots of genus $2$ up to seven crossings.}
\url{http://dmf.unicatt.it/paolini/handlebodies/}

\nada{
  
\bibitem{RamGovKau:94} P. Ramadevi, T. R. Govindarajan, R. K. Kaul,
\textit{Chirality of knots $9_{42}$ and $10_{71}$ and Chern-Simons theory},
Mod. Phys. Lett. A9 (1994), 3205--3218.

\bibitem{Ril:71} R. Riley,
\textit{Homomorphisms of knot groups on finite groups},
Math. Comp. {\bf 25} (1971), 603--619
  
  
\bibitem{Rol:03} D. Rolfsen,
\textit{Knots and Links}, AMS Chelsea Publishing, vol.364, 2003.


 
\bibitem{RouSan:82} C. P. Rourke, B. J. Sanderson,
\textit{Introduction to Piecewise-Linear Topology},
Springer-Verlag, Berlin-New York, (1982).

 
\bibitem{Sch:27} O. Schreier 
\textit{Die Untergruppen der freien Gruppe},
Abh. Math. Semin. Uni. Hamburg {\bf 5} (1927), 161--183. 





\bibitem{Sma:59} S. Smale
\textit{Diffeomorphisms of the 2-Sphere},
Proc. Amer. Math. Soc.
{\bf 10} (4) (1959) 621--626. 

Proc. Amer. Math. Soc.


\bibitem{SteThu:97} W. P. Thurston,
\textit{Three-dimensional Geometry and Topology, Volume 1}, Princeton University Press, (1997).
 
\bibitem{Sti:87} J. Stillwell,
\textit{In Papers on Group Theory and Topology, Appendix: The Dehn-Nielsen Theorem}, Springer-Verlag, New York Inc., (1987).

\bibitem{Sta:65} J. R. Stallings, 
\textit{A topological proof of Grushko's theorem on free products } 
Math. Z. {\bf 90} (1965), 1--8. 
  
  

\bibitem{Suz:75} S. Suzuki,
\textit{On surfaces in 3-sphere: prime decompositions},
Hokkaido Math. J. {\bf 4} (1975), 179--195. 
 
\bibitem{Swa:70} G. A. Swarup,
\textit{Some properties of 3-manifolds with boundary},
Quart. J. Math. Oxford {\bf 21} (1970), 1--23.
 
\bibitem{Thu:97} W. P. Thurston,
\textit{Three-Dimensional Geometry and Topology},
Princeton University Press, Princeton, NJ, (1997). 
  
\bibitem{Tsu:70} Y. Tsukui,
\textit{On surfaces in 3-space}, 
Yokohama Math. J. {\bf 18} (1970), 93--104  


\bibitem{Tsu:75} Y. Tsukui,
\textit{On a prime surface of genus $2$ and homeomorphic splitting of $3$-sphere},
The Yokohama Math. J. {\bf 23} (1975), 63--75.
}

\bibitem{Wan:22} Y.-S. Wang,
\textit{Unknotting annuli and handlebody-knot symmetry},
Topol. Appl. {\bf 305}(1) (2022), 107884.


\bibitem{Wan:24} Y.-S. Wang,
\textit{JSJ decomposition for handlebody-knots}, 
J. Math. Soc. Japan {\bf 76}(4) (2024), 1049--1085.


\nada{
\bibitem{Wal:67}F. Waldhausen,
\textit{Gruppen mit Zentrum und 3-dimensionale Mannigfaltigkeiten},
Topology {\bf 6} (1967), 505--517.


\bibitem{Wal:68i} F. Waldhausen,
\textit{Heegaard-Zerlegungen der 3-Sph\"are}, 
Topology {\bf 7} (1968), 195--203. 

 
\bibitem{Wal:68ii} F. Waldhausen,
\textit{On irreducible 3-manifolds which are sufficiently large},
Ann. of Math. {\bf 87} (1968), 56--88.
} 
  
\bibitem{Whi:32} H. Whitney,
\textit{Congruent graphs and the connectivity of graphs},
Am. J. Math. {\bf 54} (1932), 150--168.


\nada{
\bibitem{Whi:87} W. Whitten, 
\textit{Knot complements and groups},
Topology {\bf 26}, Issue 1, (1987), 41--44.
 
   

\bibitem{Yet:89} D. N. Yetter,
\textit{Category theoretic representations of knotted graphs in $S^3$},
Adv. Math. {\bf 77} (1989), 137--155.
}
\bibitem{Zee:63} E. C. Zeeman,
\textit{Seminar on combinatorial topology}, 
Notes I.H.E.S. and Warwick (1963).

  

%
\end{thebibliography}
\end{document}